\documentclass[12pt]{extarticle}
\usepackage{amsmath, amsthm, amssymb, color}
\usepackage[colorlinks=true,linkcolor=blue,urlcolor=blue]{hyperref}
\usepackage{graphicx}
\usepackage{blkarray}
\usepackage{caption}
\usepackage{mathtools}
\usepackage{enumerate}
\usepackage{verbatim}
\usepackage{tikz,tikz-cd,tikz-3dplot}
\usepackage{amssymb}
\usepackage{breakurl}
\usetikzlibrary{matrix}
\usetikzlibrary{arrows}
\usepackage{algorithm}
\usepackage[noend]{algpseudocode}
\usepackage{caption}
\usepackage[normalem]{ulem}
\usepackage{subcaption}
\oddsidemargin \evensidemargin
\usepackage{makecell}
\usepackage{array}

\newtheorem{theorem}{Theorem}
\newtheorem{proposition}[theorem]{Proposition}
\newtheorem{lemma}[theorem]{Lemma}
\newtheorem{corollary}[theorem]{Corollary}
\newtheorem{algo}[theorem]{Algorithm}
\theoremstyle{definition}

\newtheorem{remark}[theorem]{Remark}
\newtheorem{conjecture}[theorem]{Conjecture}

\newtheorem{example}[theorem]{Example}

\numberwithin{theorem}{section}
\theoremstyle{plain}

\newcommand{\PP}{\mathbb{P}}
\newcommand{\RR}{\mathbb{R}}
\newcommand{\QQ}{\mathbb{Q}}
\newcommand{\CC}{\mathbb{C} }

\DeclareMathOperator{\trop}{trop}
\DeclareMathOperator{\SGr}{SGr}
\DeclareMathOperator{\Gr}{Gr}
\DeclareMathOperator{\Dr}{Dr}
\DeclareMathOperator{\val}{val}

\newtheorem{innercustomthm}{Theorem}
\newenvironment{customthm}[1]
{\renewcommand\theinnercustomthm{#1}\innercustomthm}
{\endinnercustomthm}

\newtheorem{innercustomalg}{Algorithm}
\newenvironment{customalg}[1]
{\renewcommand\theinnercustomalg{#1}\innercustomalg}
{\endinnercustomalg}

\title{\bf Self-dual matroids from canonical curves}

\author{Alheydis Geiger, Sachi Hashimoto,
    Bernd Sturmfels and Raluca Vlad}

\date{}
\begin{document}
\maketitle

\begin{abstract} \noindent
Self-dual configurations of $2n$ points in
a projective space of dimension $n-1$
were studied by
Coble, Dolgachev--Ortland, and
Eisenbud--Popescu.
We examine the self-dual matroids and
self-dual valuated matroids defined by
such configurations, with a focus on those arising from
hyperplane sections of canonical curves. 
These objects are
parametrized by the self-dual Grassmannian and its tropicalization.
We tabulate all self-dual matroids up to rank 5 and investigate their realization spaces.
Following Bath, Mukai, and  Petrakiev, we explore algorithms for recovering a curve from the configuration.
A detailed analysis is given for self-dual
 matroids arising from graph~curves.
\end{abstract}

\section{Introduction}

A configuration of $2n$ points that span projective space $\PP^{n-1}$
is represented as the columns of
an $n \times 2n$ matrix $X$ of rank $n$.
The matrix and the configuration are called  {\em self-dual} if
\begin{equation}
 \label{eq:XDX}
  \qquad \qquad X \cdot \Lambda \cdot X^T \, = \, 0 \qquad \hbox{ for some diagonal matrix $\Lambda$ with nonzero entries.}
  \end{equation}
  This is a system of $\binom{n+1}{2}$ linear equations in $2n$ unknowns, namely the entries in $\Lambda$.
 The matrix for this system has $2n$ columns, one for each point, and  it has
   $\binom{n+1}{2}$ rows, one for each quadratic monomial.
The entries are the evaluations of quadratic monomials at the~points.
The rank of 
 this matrix is less than $2n$ whenever
 $X$ is self-dual.
  We always assume $n \geq 2$.

Suppose $n=2$. Then
$X = (x_{ij})$ is a $2 \times 4$ matrix, representing four points in~$\PP^1$.
Setting
$\Lambda = {\rm diag}(\lambda_1,\lambda_2,\lambda_3,\lambda_4)$, the constraint (\ref{eq:XDX})
becomes the following system of linear equations:
\begin{equation} 
\label{eq:sys4}
\begin{small}
\begin{pmatrix} 
x_{11}^2 & x_{12}^2 & x_{13}^2 & x_{14}^2 \\
x_{11} x_{21} & x_{12} x_{22} & x_{13} x_{23} & x_{14} x_{24} \\
x_{21}^2 & x_{22}^2 & x_{23}^2 & x_{24}^2 
\end{pmatrix} \begin{pmatrix} \lambda_1 \\ \lambda_2 \\ \lambda_3 \\ \lambda_4 \end{pmatrix} \,\, = \,\,
\begin{pmatrix} 0 \\ 0 \\ 0 \end{pmatrix} . \end{small}
\end{equation}
If the four points are distinct (i.e.~all six $2 \times 2$ minors of $X$ are nonzero) then,
up to scaling, the system
 (\ref{eq:sys4}) has a unique solution, whose four
coordinates $\lambda_i$ are nonzero.
We conclude that every configuration of four distinct points in $\PP^1$ is self-dual. This fails when two of 
the four points collide, but still self-duality can be extended  to the compact moduli space~$\overline{\mathcal{M}}_{0,4}$.

For $n \geq 3$, self-duality imposes $\binom{n-1}{2}$ independent constraints on
point configurations. For instance, for $n=4$ it imposes three constraints on
eight points in $\PP^3$. This leads us to
  {\em Cayley octads} \cite{DO, PSV}, i.e.~configurations obtained by
  intersecting three quadratic surfaces~in~$\PP^3$.

We now explain the paper title.
Matroid theory and algebraic geometry have had  exciting
interactions in recent years. We contribute to
that thread here. The matrix $X$ specifies a matroid $M$ of rank $n$
on the ground set $[2n] = \{1,2,\ldots,2n\}$.
The constraint~(\ref{eq:XDX}) implies that $M$ is {\em self-dual}.
For us, this means that an $n$-element subset $I$ of $[2n]$ is a basis of $M$ 
if and only if its  complement $[2n] \backslash I$ is also a basis of $M$.
Our main objects of study  are
 self-dual matroids, regardless of whether they are
realizable or not. For first examples, take~$n=3$.  The matroid with two non-bases $123$ and $456$ is self-dual,
but the one with only one non-basis is not self-dual.
This mirrors the algebraic constraints in  (\ref{eq:XDX}) and  (\ref{eq:XDX2}). In our setting,
the six points in $\PP^2$ given by $X$ are distinct, they lie on a conic, and
no four are on a line.  

Our title refers to the canonical
embedding of an algebraic curve of genus $g=n+1$.
This curve lives in the projective space $\PP^{g-1}$ where it has
degree $2n=2g-2$. A canonical divisor $X$ on that curve is
the intersection with a hyperplane $H \simeq \PP^{g-2}$.
Assuming $H$ to be general, $X$ is
 a configuration of $2g-2$ distinct points in $\PP^{g-2}$.
 We claim that $X$ is self-dual. To see this, let
  $D$ be any subset of $g-1$ points and $X-D$ the
complementary set of $g-1$~points. Both are viewed as 
effective divisors on the curve.


Since hyperplanes in $\PP^{g-1}$ containing $D$ are in one-to-one correspondence with sections of the canonical divisor $X$ vanishing along $D$, we have that
$h^0(X-D)  =  g-1 - \dim \overline{D},$
where $\overline{D}$ denotes the linear span of the points in $D$. And similarly
$h^0(D) =  g-1 - \dim \overline{X - D}.$
The Riemann--Roch theorem tells us that
$$h^0(D) \, -\, h^0(X - D) \,\, = \,\, {\rm deg}(D) - g + 1 \,\, = \,\, 0.$$
This means that $\dim \overline{D} = \dim \overline{X - D}$.
Hence the configuration $X$ is self-dual in the sense of (\ref{eq:XDX}), and its matroid is self-dual.
Viewed from this angle,
 self-duality of matroids
 is a combinatorial shadow of the Riemann--Roch Theorem.

\smallskip

We now highlight four main results of our paper. 
The first concerns matroids of rank $5$.

\begin{customthm}{\ref{thm:1042}}
Up to isomorphism, there are $1042$ simple self-dual matroids $M$ of rank~$5$. 
Of these, $346$ are not realizable over $\CC$. For the remaining $696$ 
matroids, the complex realization spaces  have dimensions
$\,0^{16}, \,1^{113},\, 2^{220}, \,3^{175}, \, 4^{83},\, 5^{27},\, 6^{34}, \, 7^4,\, 8^{13},\, 9^1, \,10^6,\, 12^2,\, 14^1, \,16^1.$ 
\end{customthm}
The exponent indicates the number of matroids with realization space of that dimension.
We provide self-dual configurations realizing almost all of these matroids, when they exist. In Example \ref{ex:interesting} we show a representable self-dual matroid that is not realizable as a self-dual point configuration.
We address the problem of passing a canonical curve through these self-dual configurations.  

\begin{customalg}{\ref{alg:genus6}}
Given $10$ self-dual points in $\PP^4_\QQ$ whose ideal $I_X$ has 
generic 
Betti table~\eqref{eq:betti6},
we 
find the ideal $I_C$ in $\QQ[x_0,\ldots,x_5]$ of a 
smooth canonical curve $C$ such that $I_C|_{x_5 = 0}=I_X$. 
\end{customalg}

We implement this algorithm, which rests on work of Bath in \cite{Bath},
in  \texttt{Macaulay2} \cite{M2} and test it on data from Theorem \ref{thm:1042}.
Conjecture \ref{conj:alggenus} concerns the correctness of 
Algorithm \ref{alg:genus6}.

We also study matroids $M_G$ arising from graph curves  \cite{BE}.
By Theorem \ref{thm:MG}, $M_G$ is~self-dual and determined by the graph $G$.
We present the complete classification for genus $g \leq 7$.

\begin{customthm} {\ref{thm:graphcurves4567}}
  For $g=4,5,6,7$, there are  $\,2,4,14,57$ distinct graph curves $C_G$.
  Their $3$-connected trivalent graphs $G$ yield $2,4,12,45$ distinct self-dual matroids $M_G$ of
  ranks $3,4,5,6$.
\end{customthm}

Finally, we turn to tropical moduli spaces of  self-dual valuated matroids. These give insight into degenerations of point configurations. Our main result concerns six points in $\PP^2$.

  \begin{customthm}{\ref{thm:dreisechs}}
  The tropical self-dual Grassmannian $\trop(\SGr(3,6))$ consists of all self-dual valuated matroids 
  of rank $3$.
   It is linearly isomorphic to the tree space $\,{\rm trop}({\rm Gr}(2,6))$, i.e.
   \begin{equation*}
   \trop(\Gr(2,6))\,\cong\, {\rm trop}({\rm SGr}(3,6))\, = \, {\rm trop}({\rm Gr}(3,6))^{\rm sd}
        \,= \,
       {\rm Dr}(3,6)^{\rm sd}.
    \end{equation*}
The superscript $(\cdot)^{\rm sd}$ denotes the fixed point locus under the tropical Hodge star map \eqref{eq:selfdual}.
  \end{customthm}

Our objects of study appeared in the literature under slightly different names and with slightly different hypotheses.
Coble~\cite{coble} used the adjective {\em associated points} for the Gale dual
of a configuration.  The configurations we call self-dual
were  {\em associated}  for Bath \cite{Bath} and 
  {\em self-associated} for Petrakiev \cite{Pet}.
 Important contributions to their study were made
by Dolgachev and Ortland in \cite[Chapter~III]{DO}.
The modern theory  appears in work 
of Eisenbud and Popescu~\cite{EP}, who developed 
 solid foundations based on concepts in commutative algebra.

In the combinatorics literature, one finds only few studies of matroids that coincide
with their own dual. Such matroids usually arise in
coding theory \cite{PG}. The matroids we call self-dual have
been known as {\em identically self-dual} (ISD)
to matroid theorists  \cite{Lind}.
A recent study of ISD matroids from the combinatorial 
viewpoint is the masters thesis of Perrot \cite{Per}.
We use the adjective {\em self-dual} instead of ISD here,
both for matroids and for configurations.

\smallskip

This paper is structured as follows. Section~\ref{sec2} is devoted to self-dual 
point configurations and  how to parametrize them. We introduce the
self-dual Grassmannian and its self-dual matroid strata.
Our approach extends work of Dolgachev and Ortland in \cite[Section III.2]{DO}.
 In Section \ref{sec3} we classify small self-dual matroids, up to rank $5$, and discuss large scale computations of their associated data, including their realization spaces and self-dual realizations.
In Section \ref{sec4} we study  canonical curves $C$ of genus $g$,
for $g\leq 10$, and discuss the lifting problem:
given a configuration $X$, find $C$ and a 
hyperplane $H$ such that $X = C \cap H$.
The setting is that of {\em Mukai Grassmannians}, whose defining ideals we show explicitly.
In Section~\ref{sec5} we consider a degenerate case of canonical curves.
{\em Graph curves} are arrangements of lines in $\PP^{g-1}$
that represent stable nodal curves \cite{BE}.  They arise from
trivalent $3$-connected graphs with $2g-2$ vertices and $3g-3$ edges.
We study the matroids  given by
canonical divisors on graph curves.
 Section \ref{sec6} concerns tropical limits of self-dual configurations.
We examine the tropical self-dual Grassmannian whose points are self-dual valuated matroids.
These represent the self-dual locus in the compactifications of Kapranov \cite{Kap}
and Keel--Tevelev~\cite{KT}.

This article relies heavily on software and data. These materials
are made available~at \url{https://mathrepo.mis.mpg.de/selfdual},
in the repository
{\tt MathRepo} at MPI-MiS~\cite{MathRepo}.

\section{The self-dual Grassmannian}
\label{sec2}

The variety of self-dual point configurations 
of $2n$ points spanning $\PP^{n-1}$
is central to our study where $n = g-1$. We present this
variety in the context of the Grassmannian
${\rm Gr}(n,2n)$.  
The Grassmannian perspective is well suited
for combinatorial and computational purposes.

Given a field $K$,
the {\em Grassmannian} ${\rm Gr}(n,2n)$ parametrizes
$n$-dimensional subspaces of $K^{2n}$.
Such a subspace is the row span of an
$n \times 2n$ matrix $X$ with linearly independent rows.
The Pl\"ucker embedding of ${\rm Gr}(n,2n) $ into $ \PP^{\binom{2n}{n}-1}$
represents this subspace by the vector $p$ of maximal minors $p_{i_1 i_2 \cdots i_n}$ of $X$,
where $1 \leq i_1 < i_2 < \cdots < i_n \leq 2n$.
The prime ideal of ${\rm Gr}(n,2n)$ has a Gr\"obner basis of
  {\em Pl\"ucker quadrics}
\cite[Theorem 3.1.7]{AIT}. 
For  $n=2$ that ideal is
$\langle p_{12} p_{34} - p_{13} p_{24} + p_{14} p_{23}  \rangle$.
For $n=3,4,5$, try
the command {\tt Grassmannian} in {\tt Macaulay2}~\cite{M2}.

\smallskip

The ${2n}$-dimensional algebraic torus $(K^*)^{2n}$ acts on ${\rm Gr}(n,2n)$
by scaling the columns of~$X$. This is right multiplication 
 $\,X \mapsto X \cdot {\rm diag}(\lambda)\,$ by the diagonal matrix
 with entries $\lambda_1,\ldots,\lambda_{2n}$.
That action lifts to Pl\"ucker coordinates as follows.
For $\lambda \in (K^*)^{2n}$ and $p \in \PP^{\binom{2n}{n}-1}$, it maps
\begin{equation}
\label{eq:torusaction} 
p \,\mapsto \,\lambda p \quad \hbox{with coordinates} \quad p_{i_1 i_2 \ldots i_n} \,\, \mapsto \,\,
\lambda_{i_1} \lambda_{i_2} \cdots \lambda_{i_n}\, p_{i_1 i_2 \ldots i_n}  .
\end{equation}

The  {\em open Grassmannian} ${\rm Gr}(n,2n)^o$ is the set of all points
 $p \in {\rm Gr}(n,2n)$ whose $\binom{2n}{n}$ Pl\"ucker coordinates are nonzero.
 The torus $(K^*)^{2n}$ acts on this $n^2$-dimensional manifold 
with one-dimensional stabilizers.
All orbits of this action are closed. We can thus define
\begin{equation}
\label{eq:Rn2n}
 \mathcal{R}(n,2n) \quad := \quad {\rm Gr}(n,2n)^o/(K^*)^{2n}. 
 \end{equation}
 This quotient  is a very affine variety of dimension $(n-1)^2$.
It parametrizes
  labeled configurations of $2n$ points in $\PP^{n-1}$,
modulo projective transformations, with
no $n$ points  in~a hyperplane.
For example, configurations of six points in $\PP^2$ are represented
 by Pl\"ucker vectors
 \begin{equation}
 \label{eq:pvec}
 p \,\,= \,\,(  p_{123},  p_{124}, p_{125},p_{126},p_{134}, p_{135},\ldots, p_{346}, p_{356}, p_{456}) 
 \end{equation}
 modulo the scaling action 
 $\,p_{ijk}\, \mapsto \,\lambda_i \lambda_j \lambda_k \,p_{ijk}\,$ we
 saw in (\ref{eq:torusaction}).
 Here $p_{ijk} \not= 0$ for all $i,j,k$.
 
Whenever affine coordinates are preferred, we represent points $X \in \mathcal{R}(n,2n)$ by matrices
\begin{equation}
\label{eq:bymatrices}
 X \quad = \quad  \begin{small}
\begin{pmatrix} 
1 & 0  & 0 & \cdots & 0 & 1  & 1 & 1 & \cdots  &1 \\
0 & 1 & 0 & \cdots & 0 & 1 & x_{2,n+2} & x_{2,n+3} & \cdots & x_{2,2n}  \\
0 & 0 & 1 & \cdots & 0 & 1 & x_{3,n+2} & x_{3,n+3} & \cdots & x_{3,2n}  \\
\vdots & \vdots & \vdots & \ddots & \vdots & \vdots & \vdots & \vdots & \cdots & \vdots \\
0 & 0 & 0 & \cdots & 1 & 1 & x_{n,n+2} & x_{n,n+3} & \cdots & x_{n,2n} 
\end{pmatrix}. \end{small}
\end{equation}
The condition for $X$ to lie in $\mathcal{R}(n,2n)$ is that all 
$n \times n$ minors are non-zero. Note that each such minor is equal,
up to sign, to a minor of some size in the $n \times n $ submatrix on the right.

In algebraic geometry, it is desirable to
compactify the configuration space
$\mathcal{R}(n,2n)$, e.g.~by extending the quotient (\ref{eq:Rn2n})
to the full Grassmannian ${\rm Gr}(n,2n)$. The standard method for this is
Geometric Invariant Theory \cite[\S II.1]{DO}.
An alternative approach is based on
Chow quotients \cite{Kap, KT}
and their combinatorial representation  using matroid subdivisions.
We will discuss this  in Section \ref{sec6}.
First, however, we turn to matroid theory (cf.~\cite[\S 13.1]{INLA}).

Let $M$ be any 
simple rank $n$ matroid on  $[2n] =\{1,2,\ldots,2n\}$.
Its matroid stratum ${\rm Gr}(M)$  consists of points $p$ in
${\rm Gr}(n,2n)$ for which $p_I = p_{i_1 i_2 \cdots i_n} $
is non-zero if and only if $I = \{i_1,i_2,\ldots,i_n\}$ is a basis of $M$.
The torus $(K^*)^{2n}$ acts with closed orbits, and we  define
$$ \mathcal{R}(M) \,\,\, := \,\,\, {\rm Gr}(M)/(K^*)^{2n}. $$
This is a very affine variety, called  the {\em realization space} of $M$.
Its elements are configurations of $2n$ points in $\PP^{n-1}$,
modulo projective transformations, where the points 
satisfy the dependency constraints imposed by $M$.
If $M$ is the uniform matroid then $\mathcal{R}(M) = \mathcal{R}(n,2n)$.

\begin{example}[$n=3$] \label{ex:36two}
Let $M$ be the matroid with two non-bases $123$ and $456$.
Then $\mathcal{R}(M)$ parametrizes pairs of collinear triples in $\PP^2$. 
This very affine surface is given~by
$$ \qquad
X \quad = \quad 
\begin{pmatrix}
1 & 0 & 1 & 0 & 1 & 1 \\  
0 & 1 & x & 0 & 1 & 1 \\
0 & 0 & 0 & 1 & 1 & y \end{pmatrix}, \qquad
\hbox{where} \,\, \,x y (1-x) (1-y) \not= 0.
$$
Geometrically, the surface $\mathcal{R}(M)$ is the affine plane $K^2$ with four 
special lines removed.
\end{example}

We next discuss the natural involution on ${\rm Gr}(n,2n)$.
Let $p^*$ denote the Pl\"ucker vector of the orthogonal complement
of the subspace with Pl\"ucker vector $p$. The map $p \mapsto p^*$,
known  as the {\em Hodge star},
 is given by the following combinatorial rule.
If $I$ is an ordered $n$-subset of $[2n] = \{1,2,\ldots,2n\}$ and
$I^c = [2n] \backslash I$ is the complementary ordered $n$-subset,~then
\begin{equation}
\label{eq:hodgestar}
 p^*_{I^c} \,\, := \,\, {\rm sign}(I,I^c) \cdot p_I. 
 \end{equation}
Here ${\rm sign}(I,I^c)$ is the sign of the permutation of $[2n]$
that sends the sequence $(1,2,\ldots,2n)$ to  the ordered sequence $(I,I^c)$.
For instance, for $n{=}3$, the Hodge star maps (\ref{eq:pvec}) to the vector
  \begin{equation}
 \label{eq:pvecdual}
 p^* \,\,= \,\,(  \,-p_{456},  \,p_{356}, -p_{346},\,p_{345},-p_{256}, \,p_{246},\,\ldots,\, p_{125}, -p_{124},\, p_{123}\,) .
 \end{equation}
 
 A point $p$ in the Grassmannian ${\rm Gr}(n,2n)$ is called {\em self-dual}
 if $p^* = \lambda p$ for some $\lambda \in (K^*)^{2n}$.
  We define the
 {\em self-dual Grassmannian}  ${\rm SGr}(n,2n)$ to be the subvariety consisting
 of all self-dual points in
 ${\rm Gr}(n,2n)$.   If we restrict to nonzero Pl\"ucker coordinates,
then we obtain the {\em open self-dual Grassmannian}
 ${\rm SGr}(n,2n)^o = {\rm SGr}(n,2n) \cap {\rm Gr}(n,2n)^o$.
 As we shall see, this very affine variety is cut out in ${\rm Gr}(n,2n)^o$
 by a system of binomial equations of degree four.

\begin{example}[$n=3$] \label{ex:cca} A point $p$ in ${\rm Gr}(3,6)^o$ is self-dual if
$p^* = \lambda p$ for some $\lambda \in (K^*)^6$. This implies
$\lambda_i \lambda_j \lambda_k = p^*_{ijk}/ p_{ijk}$.
The square-free monomials $\lambda_i \lambda_j \lambda_k$
satisfy certain quadratic binomial equations. This is known 
in combinatorial commutative algebra as the 
toric ideal of the hypersimplex $\Delta(3,6) $.
By substituting $\lambda_i \lambda_j \lambda_k = p^*_{ijk}/ p_{ijk}$
into these binomials, we obtain the quartic
equation that defines ${\rm SGr}(3,6)^o$. Here is the explicit computation:
$$ 0 \,\, = \,\,
(\lambda_1 \lambda_2 \lambda_3) ( \lambda_1 \lambda_4 \lambda_5) -
(\lambda_1 \lambda_2 \lambda_4) ( \lambda_1 \lambda_3 \lambda_5) \,\,=\,\,
\frac{p^*_{123}}{p_{123}} \frac{p^*_{145}}{p_{145}} -
\frac{p^*_{124}}{p_{124}} \frac{p^*_{135}}{p_{135}} .
$$
Clearing denominators gives the formula
\cite[eqn~(3.4.9)]{AIT}
 for six points in $\PP^2$ to lie on a~conic:
\begin{equation}
\label{eq:sixpointsonconic} 
0 \,\, = \,\,
p_{123} p_{145} p^*_{124} p^*_{135}\, -\,
p_{124} p_{135} p^*_{123} p^*_{145} \quad = \quad
p_{123} p_{145} p_{356} p_{246} \, - \, p_{124} p_{135} p_{456} p_{236} .
\end{equation}
The self-dual Grassmannian 
${\rm SGr}(3,6)$ is the divisor in ${\rm Gr}(3,6)$ defined by this equation.
The right hand side of (\ref{eq:sixpointsonconic}) equals the
determinant of the $6 \times 6$ matrix below, which arises from~(\ref{eq:XDX}):
\begin{equation}
\label{eq:XDX2} \begin{small}
\begin{pmatrix}
x_{11}^2  & x_{12}^2  & x_{13}^2  & x_{14}^2  & x_{15}^2  & x_{16}^2  \\
x_{21}^2  & x_{22}^2  & x_{23}^2  & x_{24}^2  & x_{25}^2  & x_{26}^2  \\
x_{31}^2  & x_{32}^2  & x_{33}^2  & x_{34}^2  & x_{35}^2  & x_{36}^2  \\
x_{11} x_{21} & x_{12} x_{22} & x_{13} x_{23} & x_{14} x_{24} & x_{15} x_{25} & x_{16} x_{26} \\
x_{11} x_{31} & x_{12} x_{32} & x_{13} x_{33} & x_{14} x_{34} & x_{15} x_{35} & x_{16} x_{36} \\
x_{21} x_{31} & x_{22} x_{32} & x_{23} x_{33} & x_{24} x_{34} & x_{25} x_{35} & x_{26} x_{36} 
\end{pmatrix}
\cdot
\begin{pmatrix} \lambda_1 \\ \lambda_2 \\  \lambda_3 \\ \lambda_4 \\  \lambda_5 \\ \lambda_6 \end{pmatrix} 
\,\,\, = \,\,\,
\begin{pmatrix} 0 \\ 0 \\ 0 \\ 0 \\ 0 \\ 0 \end{pmatrix}.
\end{small}
\end{equation}
Our definition of self-duality requires that each $\lambda_i$ is non-zero and that the six points are distinct and span $\PP^2$.
This implies that they lie on a conic and no four are on a line.
\end{example}

Restricting the quotient (\ref{eq:Rn2n}) to the self-dual Grassmannian, we get the very affine~variety
\begin{equation}
\label{eq:Sn2n}
 \mathcal{S}(n) \,\,\, := \,\,\, {\rm SGr}(n,2n)^o/(K^*)^{2n}. 
 \end{equation}
Its elements correspond
to self-dual configurations
in general linear position  in $\PP^{n-1}$, considered modulo projective transformations.
We call $ \mathcal{S}(n)$ the  {\em self-dual configuration space}.

\begin{example}[$n=3,4$] \label{ex:34}
The configuration space $\mathcal{R}(3,6) $
is $4$-dimensional. Its elements are 6-tuples in general linear position in $\PP^2$.
Its subvariety $\mathcal{S}(3)$ has codimension $1$ and is defined by
the quartic (\ref{eq:sixpointsonconic}).
This very affine threefold  parametrizes  6-tuples lying on a~conic.

The configuration space $\mathcal{R}(4,8)$ is $9$-dimensional and very affine. Its subvariety 
$\mathcal{S}(4)$ parametrizes  {\em Cayley octads}, i.e.~intersections of three quadrics in $\PP^3$.
See \cite[pages 48 and 107]{DO}. 
The subvariety  $\mathcal{S}(4)$ has codimension $3$ in
$\mathcal{R}(4,8)$. It is cut out by
$21$ binomial quartics that are derived like (\ref{eq:sixpointsonconic}).
These binomials are displayed explicitly in \cite[Proposition 7.2]{PSV}.
\end{example}

\begin{theorem}[{\cite[Theorem~4, p.~51]{DO}}] \label{thm:SOn}
Assume that we are over an algebraically closed field $K$.
The self-dual configuration space
$\mathcal{S}(n)$ has
a  birational parametrization by
the rotation group ${\rm SO}(n)$. Hence
$\mathcal{S}(n)$ is rational and  has dimension $\binom{n}{2}$.
\end{theorem}

\begin{corollary} \label{cor:rationaldim}
The self-dual Grassmannian
${\rm SGr}(n,2n)$ is  rational 
of dimension $\binom{n}{2}+2n-1$.
\end{corollary}

The following shows how we can create sample points in $\SGr(n,2n)$ from the fact that $\mathrm{SO}(n)$ is finite to one to the self-dual configuration space $\mathcal{S}(n)$.

Let $p \in {\rm SGr}(n,2n)^o$. Then $p$ corresponds to an $n \times 2n$ matrix $X = (X_1| X_2)$ where $X_1, X_2$ are invertible $n \times n$ matrices. The self-duality condition~(\ref{eq:XDX}) 
is equivalent to $X_1 \Lambda_1 X_1^T =  X_2 \Lambda_2 X_2^T$ for invertible 
 diagonal $n \times n$ matrices $\Lambda_1$ and $\Lambda_2$. These diagonal 
matrices possess a square root over $ \overline{K}=K$. The identity above is equivalent 
to $\Lambda_1 = (X_1^{-1} X_2) \Lambda_2 (X_1^{-1} X_2)^T$, and hence to
$\Lambda_1^{-1/2} (X_1^{-1} X_2) \Lambda_2^{1/2} \,\Lambda_2^{1/2} (X_1^{-1} 
X_2)^T \Lambda_1^{-1/2} = {\rm Id}_n$. This last identity says that, up to sign, the  
matrix $R = \Lambda_1^{-1/2} (X_1^{-1} X_2) \Lambda_2^{1/2}$ is in ${\rm SO}(n)$.
The given matrix $X$ has the same row span as 
$({\rm Id}_n \, | \, X_1^{-1} X_2)$, which has the same row span as
\begin{equation} \label{eq:alsop}
 \Lambda_1^{-1/2} ({\rm Id}_n \,| \, X_1^{-1} X_2) \,\,=\,\,( \Lambda_1^{-1/2}
  \,| \,R \Lambda_2^{-1/2})   \,\, = \,\, ( {\rm Id}_n \,| \, R ) \cdot {\rm diag}(\lambda), \end{equation}
where ${\rm diag}(\lambda)$ consists of
the blocks $\Lambda_1^{-1/2}$ and $\Lambda_2^{-1/2}$. 
(Note that the choices of square root will give a finite-to-one map from $\mathrm{SO}(n)$ to $\mathcal{S}(n)$.)
The matrix (\ref{eq:alsop}) represents~$p$. 

Conversely, consider $R$ in the open subset of ${\rm SO}(n)$ defined by matrices for which the maximal minors of  $ ({\rm Id}_n \,| \, R )$ do not vanish. Then,  $ ({\rm Id}_n \,| \, R )$  determines a point $p \in\mathcal{S}(n)$.

For arbitrary $n$,
this
 allows us to create
random points $X$ in
the self-dual Grassmannian ${\rm SGr}(n,2n)$. 
 We work over $\RR$ and assume 
a natural  probability distribution on~$\RR^n$.

\color{black}

\begin{algo}[Sampling from ${\rm SGr}(n,2n)$ and $\mathcal{S}(n)$] \label{alg:SOn}
Pick successively  vectors
$r_1,r_2,\ldots,r_n $ in $\mathbb{R}^n$  such that
$r_{i+1}$ is perpendicular to $r_1,\ldots,r_i$ for all $i$.
Form the $n \times n$ matrix $R = (r_1,\ldots,r_n)$.
The row span of the matrix $X = (\,{\rm Id}_n \,|\, R\,)$ is
a generic point in the self-dual Grassmannian  ${\rm SGr}(n,2n)$. With probability one, the maximal minors of $X$ are nonzero, so its image modulo
the torus $(\RR^*)^{2n}$ is a generic point in the self-dual configuration space $\mathcal{S}(n)$.
\end{algo}

In the special case $n=4$, it is known that
any general configuration of seven points in $\PP^3$
can be completed uniquely to a 
Cayley octad. 
In symbols,
there is a birational isomorphism
\begin{equation}
\label{eq:cayoctmap1}
\gamma \, : \, \mathcal{R}(4,7) \,\simeq \, \mathcal{S}(4).
\end{equation}

A formula on a local chart is given in \cite[Proposition 7.1]{PSV}.
We now express this in  Pl\"ucker coordinates $p_I$.
The idea is to start with 
quadruples $I$ in $\{1,2,3,4,5,6,7\}$. We~create a point in $\SGr(4,8)^o$ by using the following identity for the Pl\"ucker coordinates $p_{I^c}$ where $8\in I^c$:
\begin{equation}
\label{eq:cayoctmap2}
p_{I^c} \,\,=\,\, \text{sign}(I^c,I)\cdot p_I^*  \,\,=\,\, \text{sign}(I^c,I)\cdot \frac{p_I}{\prod_{i \in I} z_i }. 
\end{equation}
Here $z_1,z_2,\ldots,z_7$ are polynomials in the Pl\"ucker coordinates that must be carefully chosen to mirror the self-duality condition $p^* =\lambda p$ and to be compatible with the torus action.

\begin{theorem} \label{thm:gamma}
The Cayley octad map $\gamma$ in (\ref{eq:cayoctmap1}) is given by the
$35$ formulas (\ref{eq:cayoctmap2}) where
$$ \begin{matrix} z_1 =  p_{1234} p_{1256} p_{1357} p_{1467} - p_{1235} p_{1246} p_{1347} p_{1567}, \,\,
z_2 =  p_{1235} p_{1246} p_{2347} p_{2567} - p_{1234} p_{1256} p_{2357} p_{2467}, \\
z_3 =  p_{1234} p_{1356} p_{2357} p_{3467} - p_{1235} p_{1346} p_{2347} p_{3567}, \,\,
z_4 =  p_{1245} p_{1346} p_{2347} p_{4567} - p_{1234} p_{1456} p_{2457} p_{3467}, \\
z_5 =  p_{1235} p_{1456} p_{2457} p_{3567} - p_{1245} p_{1356} p_{2357} p_{4567}, \,\,
z_6 =  p_{1246} p_{1356} p_{2367} p_{4567} - p_{1236} p_{1456} p_{2467} p_{3567}, \\
z_7 = p_{1237} p_{1457} p_{2467} p_{3567} - p_{1247} p_{1357} p_{2367} p_{4567}.
 \end{matrix}
 $$
\end{theorem}

\begin{proof}
The proof is a direct computation, which we carried out in {\tt maple}.
We applied the formula (\ref{eq:cayoctmap2}) to an arbitrary point in ${\rm Gr}(4,7)$.
This point was represented by a $4 \times 7$ matrix that contains
a $4 \times 4$ identity matrix and whose remaining $12$ entries are variables.
The result is a point in $\PP^{69}$ whose coordinates $p_I$
are rational functions in these variables. We verified that this point lies in
 ${\rm Gr}(4,8)$ by substituting those $70$ rational functions into
 the $721$ Pl\"ucker quadrics that cut out ${\rm Gr}(4,8)$.
 We similarly verified that they satisfy the $21$ quartics
 in \cite[Proposition 7.2]{PSV}. This ensures that the point lies in
 the self-dual Grassmannian  ${\rm SGr}(4,8)$.
 
The map given by \eqref{eq:cayoctmap2} and the equations of the $z_i$ above coincides with the Cayley octad map $\gamma$ as in \cite[Proposition 7.1]{PSV} and \eqref{eq:cayoctmap1}  up to torus action on a dense subset of $\Gr(4,7)$. This is due to the fact that the eighth point of a Cayley octad is uniquely determined by its 7 points \cite{PSV}. 
 We conclude that,
modulo $(K^*)^8$, our formula 
realizes the rational map~$\gamma $.
\end{proof}

\begin{remark}
The rational map $\gamma$ is a morphism on the open set
where $z_1,\ldots,z_7$ are all nonzero. This makes precise which
``general configurations'' of $7$ points can be completed into a Cayley octad. The condition $z_i = 0$ means that the projection 
to $\PP^2$ whose center is the $i$th point maps the
other $6$ points onto a conic.
Note that $\{z_1=\cdots=z_7=0\}$ is the codimension two locus of
 $7$-tuples in $\PP^3$ that lie on a twisted cubic curve; see \cite[Example~2.9]{CS}.
\end{remark}

We return to the  setting of matroids.
Let $M$ be a simple matroid of rank $n$ on $[2n]$.
The {\em self-dual realization space}
$\mathcal{S}(M)$ is the subset of $\mathcal{R}(M)$
consisting of all points $X$ that are self-dual in the sense of
(\ref{eq:XDX}).  If $\mathcal{S}(M)$ is non-empty, then
$M$ is a self-dual matroid,~i.e.~$M = M^*$ holds.
The converse statement is not true, as we will see in
Example \ref{ex:interesting}.
If $M$ is the uniform matroid then
$\mathcal{S}(M)  = \mathcal{S}(n)$.
The inclusion $\mathcal{S}(M) \subseteq \mathcal{R}(M)$ 
can be strict or it can be an equality. 
We saw in Example \ref{ex:34} that it is
strict when $M$ is the uniform matroid $U_{3,6}$ or $U_{4,8}$.

\begin{example}[$n=3$] \label{ex:selfdual36}
The matroid $M$ in Example \ref{ex:36two} satisfies
$\mathcal{S}(M) = \mathcal{R}(M)$. Indeed,~set
$$
\lambda_1 = (1-y)(1-x),\,
\lambda_2 =  x(1-x)(y-1),\,
\lambda_3 =  y-1,\,
\lambda_4 = xy(1-y),\,
\lambda_5 = -xy,\,
\lambda_6 = x.\,
$$
These quantities are nonzero for all $(x,y) \in \mathcal{R}(M)$ and they satisfy
the matrix equation (\ref{eq:XDX}).
\end{example}

We now turn to the equations that 
cut out $\mathcal{S}(M)$ as a subvariety of $\mathcal{R}(M)$.
These are quartic binomials, generalizing those in Examples \ref{ex:cca} 
and \ref{ex:34}. Consider the {\em matroid polytope} $P_M = {\rm conv}\{e_I: I \,\,
\hbox{basis of} \,\,M \}$.
Its {\em toric ideal} $\,\mathcal{I}_M$ lives in a polynomial ring with
one variable $p_I$ for each basis $I$ of $M$, and it has Krull
dimension $n$, provided $M$ is connected.
For definitions and details see \cite[\S 13.2]{INLA}.
A famous conjecture, due to Neil White, states that
$\mathcal{I}_M$ is generated by quadrics \cite[Conjecture 13.16]{INLA}.
This holds in the Laurent polynomial~ring.

\begin{proposition} \label{prop:toricideal}
Let $M$ be a self-dual matroid. Fix quadratic binomials that generate
the toric ideal $\,\mathcal{I}_M$ in the Laurent polynomial ring. For each 
basis $I$ of $M$, replace its variable  with $p^*_I / p_I$ in these quadrics.
Clearing denominators gives quartic binomials in the
Pl\"ucker coordinates indexed by bases of $M$.
These quartics define $\mathcal{S}(M)$ as a subvariety of $\mathcal{R}(M)$.
\end{proposition}

\begin{proof}
This process eliminates the unknowns in $\lambda$ from the equation $p^* = \lambda p$.
Namely, we write $\lambda_{i_1} \lambda_{i_2} \cdots \lambda_{i_n} = 
p^*_{i_1 i_2 \cdots i_n} / p_{i_1 i_2 \cdots i_n}$, and we substitute these
ratios into the toric ideal $\mathcal{I}_M$, as in Example \ref{ex:cca}. As the toric ideal $\mathcal{I}_M$ describes the relations among the products $\prod_{i\in I}\lambda_{i}$, where $I$ is a basis of $M$, the 
resulting ideal defines $\mathcal{S}(M)$ set-theoretically.
\end{proof}

\begin{example}[Non-V\'amos matroid] \label{ex:nonvamos}
Fix $n=4$ and let $M$ be the matroid with  non-bases
$$ 1256,\,\, 1357, \,\,1458, \,\,2367, \,\,2468, \,\,3478. $$
This self-dual matroid encodes four pairs of points $15,26,37,48$
that span  concurrent lines.
The realization space $\mathcal{R}(M)$ is $4$-dimensional, very affine, and given in local coordinates by
$$ X \,\, = \,\, \begin{small} \begin{pmatrix}
1 & 0 & 0 & 0 & a & 1 & 1 & 1 \\
0 & 1 & 0 & 0 & 1 & b & 1 & 1 \\
0 & 0 & 1 & 0 & 1 & 1 & c & 1 \\
0 & 0 & 0 & 1 & 1 & 1 & 1 & d \\
\end{pmatrix}, \end{small}
$$
where $a,b,c,$ and $d$ are neither $0$ nor $1$. The self-dual realization $\mathcal{S}(M)$ is a divisor in $\mathcal{R}(M)$. It is a consequence of~(\ref{eq:XDX}) that $\mathcal{S}(M)$ is the very affine threefold given by the equation
\begin{equation}
\label{eq:vamoslocal} 2 a b c d \,-\, abc - abd - acd-bcd\,+\,a+b+c+d \,-\,2 \,\, = \,\, 0 .
\end{equation}

We compute equations in the $64$ Pl\"ucker coordinates $p_{ijkl}$ with the
algorithm in Proposition~\ref{prop:toricideal}.
The toric ideal $\mathcal{I}_M$ is minimally generated by
$636$ quadratic binomials. Substituting $p^*_{i_1i_2i_3i_4} /p_{i_1 i_2i_3i_4}$
for its unknowns, and clearing denominators, we obtain $312$ distinct
quartic binomials. Only four
are linearly independent modulo the $721$ Pl\"ucker quadrics, which are
found by setting  $p_{1256} = p_{1357}= \cdots = p_{3478} = 0$ in those for ${\rm Gr}(4,8)$. So, $\mathcal{S}(M)$ is cut out by
\begin{equation}
\label{eq:vamosglobal}
\begin{matrix} p_{1234} p_{1356} p_{2578} p_{4678} - p_{1235} p_{1346} p_{2478} p_{5678},\,\,
 p_{1234} p_{1257} p_{3568} p_{4678}-p_{1235}p_{1247} p_{3468} p_{5678},\, \\
 p_{1234} p_{1456} p_{2578} p_{3678} - p_{1245} p_{1346} p_{2378} p_{5678} , \,\,
 p_{1234} p_{1267} p_{3568} p_{4578}-p_{1236} p_{1247} p_{3458} p_{5678}, \end{matrix}
 \end{equation}
inside $\mathcal{R}(M)$. The four Pl\"ucker quartics
  in (\ref{eq:vamosglobal}) are equivalent to the local equation in (\ref{eq:vamoslocal}).
\end{example}

\section{Small self-dual matroids}
\label{sec3}

The aim of this section is to tabulate all small self-dual matroids and compute their associated data.
We restrict to simple matroids, i.e.~loops or parallel elements
are not allowed. Representable simple matroids of rank $n$ correspond to configurations of
$2n$ distinct points in $\PP^{n-1}$. Recall that a rank $n$ matroid $M$
on $[2n]= \{1,2,\ldots,2n\}$ is self-dual if $M = M^*$. This means that an $n$-set $I \subset [2n]$ is a basis of $M$ if and only if $[2n] \backslash I$ is a basis of $M$.  
In \cite{Lind, Ox, Per} this property is called identically self-dual.
  We begin this section with a combinatorial result.

\begin{proposition} \label{prop:stable}
  For a self-dual matroid $M$ of rank $n$, the following four are equivalent: \vspace{-0.1cm}
  \begin{enumerate} 
  \item[(1)] The matroid $M$ is connected. \vspace{-0.2cm}
\item[(2)]  The matroid polytope $P_M$ has the maximal dimension $2n-1$. \vspace{-0.2cm}
\item[(3)] The point $\frac{1}{2} \sum_{i=1}^{2n} e_i$ is in the interior
of $P_M$, when taken in $\bigl\{ \sum_{i=1}^{2n} x_i = n \bigr\} \simeq \RR^{2n-1}$. \vspace{-0.2cm}
\item[(4)] Every proper subset $A$ of $\,[2n]$ satisfies $\,{\rm rank}_M(A) > |A|/2$. \vspace{-0.1cm}
  \end{enumerate}
The self-dual matroid $M$ is called {\em stable} if any and hence all of these conditions are met.
\end{proposition}

\begin{proof}
Conditions (1) and (2) are equivalent for all matroids, by \cite[Proposition 2.4]{FS}.
Clearly, (3) implies (2). The implication from (2) to (3) requires self-duality:
the vertex set of $P_M$ is closed under self-duality, so its average,
which lies in ${\rm int}(P_M)$, equals
$\frac{1}{2} \sum_{i=1}^{2n} e_i$. 
The equivalence of (3) and (4) uses the inequality representation of the matroid polytope
$$ P_M \,\,\,= \,\,\, \biggl\{ x \in \RR^{2n} \,\,:\,\,
\sum_{i=1}^{2n} x_i = n \,\,\, {\rm and} \,\, \sum_{i \in A} x_i \leq {\rm rank}_M(A) \,\,\,
\hbox{for all $A \subset [2n]$} \biggr\}. $$
See \cite[Proposition 2.3]{FS}. It suffices to let $A$ run over the flats of $M$.
Condition (4) says that all inequalities hold strictly at
$\frac{1}{2} \sum_{i=1}^{2n} e_i$. This means the point is in the interior of $P_M$.
\end{proof}

Geometrically, stability
  means that  the points are distinct, no four points lie on a line, no six points lie on
a plane, and so on.  Dolgachev and Ortland  \cite[Remark 3, page 47]{DO}
state that ``the assumption of stability ... is essential'' for their results on self-dual configurations.
It therefore makes sense to exclude matroids that are not stable.  If $X$ is a configuration of $2n$ distinct points that is stable, then $X$ is self-dual if and only if $X$ is Gorenstein;
see \cite[Theorem~7.3]{EP} and \cite[Proposition~1.1]{EP2}. We will 
mostly restrict to stable self-dual matroids.

\begin{example}[$n=3$]
Up to isomorphism, there are
two simple self-dual matroids of~rank~$3$. They are both stable.
The first is the uniform matroid $M_1$, with no non-basis.
Its realization spaces satisfy $\mathcal{S}(M_1) \subsetneq \mathcal{R}(M_1)$.
The second matroid $M_2$ has two non-bases, say
$123$ and $456 $. It satisfies $\mathcal{S}(M_2) = \mathcal{R}(M_2)$.
This is the very affine surface 
 in Examples~\ref{ex:36two}~and~\ref{ex:selfdual36}.
\end{example}

Perrot  \cite[Figure 2.4]{Per} listed $12$ self-dual matroids of rank $4$  with realizations. Starting from his results and using the collection in \texttt{polyDB} \cite{Paf}, we complete the analysis of rank $4$.

\begin{proposition} \label{prop:rank4}
Up to isomorphism, there are precisely $13$ simple self-dual matroids of rank $4$, of which $12$ are stable. The complete data of their realization spaces and self-dual realization spaces is collected 
in our {\tt MathRepo} page
 \cite{MathRepo}, and it is summarized in Table \ref{tab:rank4}. 
\end{proposition}

\begin{table}[h]
  \begin{tabular}{llcc}
    Label & nonbases not containing $8$ & $\!\!\!\!\dim(\mathcal{R}(M))$ & $\dim(\mathcal{S}(M))$\\
    \hline
    
   \texttt{4.0.a} & \{ \}  & 9&6
   \\
    \texttt{4.2.a} &  \{1234\} & 7&5  \\
    \texttt{4.4.a} & \{1234, 1257\} & 5&4 \\
    \texttt{4.6.a} & \{1234, 1257, 2467\}  &3 &3\\
    \texttt{4.6.b} & \{1256, 1357, 2367\} & 4&  3\\
    \texttt{4.8.a} & \{1247, 1357, 2367, 4567\}  & 2&2 \\
    \texttt{4.8.b} & \{1267, 1347, 1356, 4567\}  & 2&2 \\
    \texttt{4.10.a} & \{1234, 1257,  2356, 2467, 3457\}  & 1&1  \\
    \texttt{4.10.b} &  \{1257, 1346, 2346, 3456, 3467\}
 & 5&4\\
    \texttt{4.12.a} &\{1267, 1357, 1456, 2356, 2457, 3467\}  & 0&0   \\
    \texttt{4.14.a}  &\{1234, 1257, 1367, 1456, 2356, 2467, 3457\}
   & $-1$ & $-1$  \\
    \texttt{4.16.a} & \{1234, 1235, 1236, 1237, 1245, 1345, 2345, 4567\} & 3&3\\
    \texttt{4.34.a} & 
    \{1234, 1235, 1236, 1237, 1245, 1246, 1247, 1345, & 2 & 2\\ &
\,\,\,1346, 1347, 1567, 2345, 2346, 2347, 2567, 3567, 4567\} \!\!
  \end{tabular}
\caption{The $13$ self-dual matroids of rank $4$. The matroid \texttt{4.34.a} is not stable.
The varieties $\overline{\mathcal{R}(M)}$ and $\overline{\mathcal{S}(M)}$  are equal when their dimensions agree.
}\label{tab:rank4}
\end{table}

\begin{proof}[Proof and discussion]
The self-dual matroids were extracted from the list of all rank $4$ matroids on $8$ elements stored in the database \texttt{polyDB} \cite{Paf}.
We give labels to each self-dual matroid in the following way: our label has the form $
\texttt{r.n.*}$ where {\tt r} is the rank of the matroid, {\tt n} is the number of nonbases, and $*$ is a letter of the alphabet, assigned arbitrarily to distinguish matroids with the same number of nonbases. For simplicity, 
 in Table \ref{tab:rank4}
 we list only the nonbases without the element $8$; self-duality recovers the rest of the~nonbases.

To find the realization space for each matroid $M$, we used Gr\"obner basis computations
in {\tt Magma}.
Note that $M$ is realizable if and only if it is realized by a $4 \times 8$ matrix 
of the form $ X = (\,\text{Id}_4 \,|\,  x_{ij} \,) $,
where $\text{Id}_4$ is the identity matrix.
Here we relabel so that $1234$ is a basis of $M$.
We now consider the ideal in $\QQ[x_{ij}]_{1 \leq i , j \leq 4}$ 
that is generated by all minors of $(x_{ij})$
that are indexed by the nonbases of $M$. This sets some of
the entries $x_{ij}$ to zero. We next consider the remaining nonzero entries $x_{kl}$
in the $4 \times 4$ matrix of unknowns, and we put $x_{kl}-1$ into the ideal
for the first such nonzero entry in each column and row.
We finally saturate this ideal by each minor that corresponds to a 
basis of $M$. The resulting ideal is denoted $I_{\mathcal{R}(M)}$.

For the self-dual realization space $\mathcal{S}(M)$, we add to  $I_{\mathcal{R}(M)}$ the constraints $X \cdot \Lambda \cdot X^T = 0$, where $\Lambda $ is a diagonal matrix with eight unknowns $\{\lambda_i\}_{1 \leq i \leq 8}$. We next saturate by $
\prod_i \lambda_i$, and then we eliminate 
$\lambda_1,\ldots,\lambda_8$.
The resulting ideal in  $\QQ[x_{ij}]_{1 \leq i , j \leq 4}$ 
is denoted $I_{\mathcal{S}(M)}$.
 Finally, we check whether $I_{\mathcal{S}(M)}$ is equal to  $I_{\mathcal{R}(M)}$.    If so, we are done and we can conclude  $\overline{\mathcal{S}(M)} = \overline{\mathcal{R}(M)}$. If not, we re-saturate by each basis minor to get the ideal  $I_{\mathcal{S}(M)}$.
 \end{proof}
 
 \begin{remark} \label{rmk:representveryaffine}
In the proof of Proposition \ref{prop:rank4} and in  Table \ref{tab:rank4},
  the very affine varieties 
$\mathcal{R}(M)$ and $ \mathcal{S}(M)$ are represented 
by their closures  $\overline{\mathcal{S}(M)}$ and $ \overline{\mathcal{R}(M)}$
in  $16$-dimensional affine space. These closures are the varieties of the ideals
 $I_{\mathcal{R}(M)}$ and  $I_{\mathcal{S}(M)}$ in the polynomial ring
$\QQ[x_{ij}]_{1 \leq i , j \leq 4}$. These varieties are 
correct closures because of the saturations we performed.
We emphasize that complete encodings of the realization spaces
$\mathcal{R}(M)$ and $\mathcal{S}(M)$ require more data
than these ideals. They also require the inequations arising from bases of $M$, and $\mathcal{S}(M)$ keeps track of the inequations from requiring the entries of $\Lambda$ to be nonzero in \eqref{eq:XDX}.
\end{remark}

Some of the matroids in Table \ref{tab:rank4} are familiar.
Type \texttt{4.10.a} is the matroid given by the vertices of
a regular $3$-cube.
Type \texttt{4.6.b} is the non-V\'amos matroid in Example~\ref{ex:nonvamos}. The matroid \texttt{4.14.a} contains the Fano plane, which is why it is not realizable in characteristic zero.
Type \texttt{4.34.a} is the direct sum $U_{2,4}\oplus U_{2,4},$ which is associated with
 two quadruples ($1234$ and $5678$) of collinear points in $3$-space. 
 This violates condition (1) in Proposition \ref{prop:stable},
     so the matroid is not stable. 

Turning to the title of this paper, we now ask 
 for realizations of these matroids 
as hyperplane sections of genus $5$ canonical curves.
We exclude the non-realizable matroid {\tt 4.14.a} and the non-stable matroid {\tt 4.34.a}.
For the other $11$ matroids, we have the following result.

\begin{proposition} \label{prop:9of12}
Precisely $\,9$ of the $11$ realizable and stable matroids in Table \ref{tab:rank4}
arise from hyperplane sections of smooth complete intersections
of three quadrics in $\PP^4$.
The other two matroids, which are labeled \texttt{4.10.b} and \texttt{4.16.a},  arise 
from smooth trigonal canonical curves. 
\end{proposition}

\begin{table}[h]
    \begin{center}
        \begin{tabular}{cl}
            Label & Three quadratic forms in five variables \\
            \hline
            \texttt{4.0.a} &     $  xy -w( 5y {-} 3 z) +v^2 ,$    $ 2xz + w(x  {-} 11y  {+} 6 z) + 2 v (x{+}y)
            ,$ $ yz +  w(7 x {-} 17y  {+} 8  z ) + v y$\\
            \hline
            \texttt{4.2.a} & $y(2x-3z)+w(y+4z) +v^2,$ $z(x+y)-w(21y-11z) +vy,$ $
            5w(x-2y+z)+vx$\\
            \hline
            \texttt{4.4.a}     &    $ y( x   - 3 z )+ w(3 y - z ) + v y,$
            $ 3z(x   - 2 y ) +w(4 y   - z  ) + v x,$
            $ w(x   - 2 y   + z  ) + v^2$\\
            \hline
            \texttt{4.6.a} &    $ y(x  - 2  z)- w ( y-  z  ) + 2v^2,$
            $ z(x   + 2 y  ) -w( y  +3z) + xv ,$ 
            $ w(2 x  - y   -3  z ) + 2xv $\\
            \hline
            \texttt{4.6.b} & $ y(3x-z)-w(11y-9z)$,
                $    z(2x-y)-w(8y-7z)+v^2$,
                $    w(x-3y+2z)+vx$    \\
            \hline
            \texttt{4.8.a} &    $ x y - z  w +vx,$
            $ 2x  z +    w (y- 3  z) + v^2,$
            $ y  z + w( 3  x  - 4  z ) - vw$     \\
            \hline
            \texttt{4.8.b} & $(x+z-2w)y+v^2,$ $(x-2y+z)w+(x+y)v$,  $2(x-y)z-3(y-z)w$\\
            \hline
            \texttt{4.10.a} &    $ y( x  -   z -   w) + v^2,$
            $ z(x   - y   -   w) + v x,$
            $ w(2x  -   y  + z ) + v z$    \\
            \hline
            \texttt{4.12.a} &    $ y (2x  - z -  w )+ v ,$
            $ z(2x   -   y  -   w) + v (w+x),$
            $  w(2 x - y   -   z) + v (z-y)$     \\
            \hline
        \end{tabular}
        \caption{Complete intersection curves give rise to nine of the self-dual matroids of rank $4$.}
        \label{tab:genus5}
    \end{center}
\end{table}

\begin{proof}
For the $9$ matroids from complete intersections, we list smooth canonical curves in Table~\ref{tab:genus5}. The matroid is realized by intersecting the curve with the hyperplane 
$ \{v {=} 0\}$ in~$\PP^4$.

We now consider the remaining two matroids \texttt{4.10.b} and \texttt{4.16.a}.
These appear in Table~\ref{tab:rank4} but not in Table \ref{tab:genus5}. For each matroid, we pick
a generic self-dual realization $X$, and examine its
homogeneous radical ideal $I_X$  in $\QQ[x,y,z,w]$. It turns out that
this ideal is minimally generated by three quadrics and two cubics,
so it is not a complete intersection. 

We show that both configurations $X$ are hyperplane sections of smooth trigonal curves in $\PP^4$.
This is done as follows. The ideal $I_X$ is generated by three quadrics and two cubics in
$\QQ[x,y,z,w]$. We write these as
the $4 \times 4$ Pfaffians of a skew-symmetric $5 \times 5$ matrix
$ \begin{small} \begin{bmatrix} 0 & A \,\\ - A^T & B\, \end{bmatrix}\end{small}$,
where $A$ is a $2 \times 3$ matrix with linear entries and
$B$ is a skew-symmetric $3 \times 3$ matrix with quadratic entries.
By adding random scalar multiples of the new variable $v$ to each matrix
entry, we obtain a Pfaffian ideal that defines a smooth
trigonal curve in $\PP^4$.
\end{proof}

\begin{remark} \label{rem:dreifuenf}
The four matroids  \texttt{4.4.a}, \texttt{4.6.a}, \texttt{4.10.b} and \texttt{4.16.a} also arise from hyperplane
sections of reducible canonical curves. 
The curves consist of
eight lines in $\PP^4$, where each line meets three other lines.
Such canonical  curves are known as {\em graph curves}, and we study them in detail in Section \ref{sec5}.
These curves are determined by their dual graphs. They are trivalent graphs with eight vertices and they determine the matroid arising from the hyperplane section of the curve.
The graphs for our four
matroids are shown in Figure~\ref{fig:graphs8}.
\end{remark}

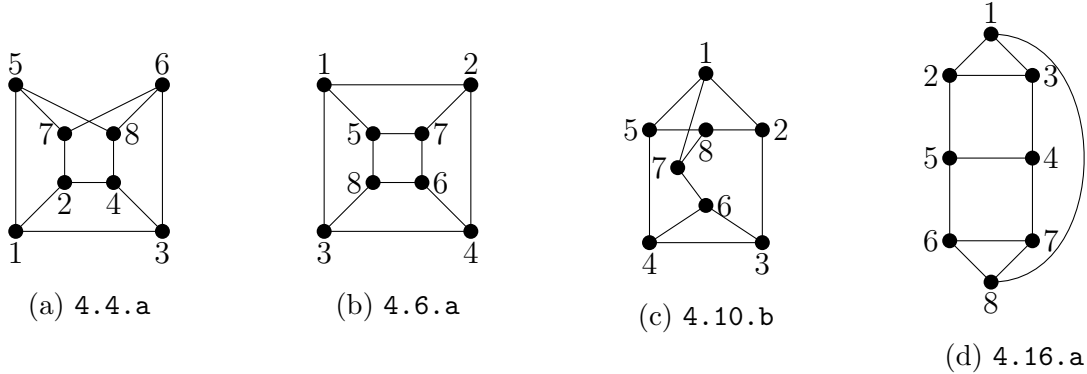
\begin{figure}[h]  
    \begin{subfigure}{.24\textwidth}\centering
    \begin{tikzpicture}[scale = 0.65]
      \coordinate (1) at (0,0);
      \coordinate (2) at (1,1);
      \coordinate (3) at (3,0);
      \coordinate (4) at (2,1);
      \coordinate (5) at (0,3);
      \coordinate (6) at (3,3);
      \coordinate (7) at (1,2);
      \coordinate (8) at (2,2);
      \foreach \i in {1,2,3,4,5,6,7,8} {
        \path (\i) node[circle, black, fill, inner sep=2]{};
      }
      \draw[black] (1) -- (3) -- (4) -- (2) -- (1) -- (5) -- (7) -- (6) -- (8) -- (5);
      \draw[black] (4) -- (8) ;
      \draw[black] (6) -- (3) ;
      \draw[black] (2) -- (7) ;
      \node[below] at (1) {1};
      \node[below] at (3) {3};
      \node[below] at (2) {2};
      \node[below] at (4) {4};
      \node[above] at (5) {5};
      \node[above] at (6) {6};
      \node[right] at (8) {8};
      \node[left] at (7) {7};
    \end{tikzpicture}
    \caption{\texttt{4.4.a}}
  \end{subfigure}
  \begin{subfigure}{.24\textwidth}\centering
    \begin{tikzpicture}[scale = 0.65]
      \coordinate (3) at (0,0);
      \coordinate (8) at (1,1);
      \coordinate (4) at (3,0);
      \coordinate (6) at (2,1);
      \coordinate (1) at (0,3);
      \coordinate (2) at (3,3);
      \coordinate (5) at (1,2);
      \coordinate (7) at (2,2);
      \foreach \i in {1,2,3,4,5,6,7,8} {
        \path (\i) node[circle, black, fill, inner sep=2]{};
      }
      \draw[black] (1) -- (2) -- (4) -- (3) -- (1) -- (5) -- (7) -- (6) -- (8) -- (5);
      \draw[black] (3) -- (8) ;
      \draw[black] (6) -- (4) ;
      \draw[black] (2) -- (7) ;
      \node[above] at (1) {1};
      \node[below] at (3) {3};
      \node[above] at (2) {2};
      \node[below] at (4) {4};
      \node[left] at (5) {5};
      \node[right] at (6) {6};
      \node[left] at (8) {8};
      \node[right] at (7) {7};
    \end{tikzpicture}
    \caption{\texttt{4.6.a}}
  \end{subfigure}
  \begin{subfigure}{.24\textwidth}\centering
    \begin{tikzpicture}[scale = 0.75]
      \coordinate (4) at (0,0);
      \coordinate (3) at (2,0);
      \coordinate (6) at (1,0.66);
      \coordinate (7) at (0.5,1.33);
      \coordinate (8) at (1,2);
      \coordinate (5) at (0,2);
      \coordinate (2) at (2,2);
      \coordinate (1) at (1,3);
      \foreach \i in {1,2,3,4,5,6,7,8} {
        \path (\i) node[circle, black, fill, inner sep=2]{};
      }
      \draw[black] (5)--(1) -- (2) -- (3) -- (4) -- (5) -- (8) -- (7) -- (6) -- (4);
      \draw[black] (3) -- (6) ;
      \draw[black] (2) -- (8) ;
      \draw[black] (1) -- (7) ;
      \node[above] at (1) {1};
      \node[below] at (3) {3};
      \node[right] at (2) {2};
      \node[below] at (4) {4};
      \node[left] at (5) {5};
      \node[right] at (6) {6};
      \node[below] at (8) {8};
      \node[left] at (7) {7};
    \end{tikzpicture}
    \caption{\texttt{4.10.b}}
  \end{subfigure}
  \begin{subfigure}{.24\textwidth}\centering
    \begin{tikzpicture}[scale = 0.55]
      \coordinate (6) at (0,0);
      \coordinate (7) at (2,0);
      \coordinate (4) at (2,2);
      \coordinate (5) at (0,2);
      \coordinate (2) at (0,4);
      \coordinate (3) at (2,4);
      \coordinate (1) at (1,5);
      \coordinate (8) at (1,-1);
      \foreach \i in {1,2,3,4,5,6,7,8} {
        \path (\i) node[circle, black, fill, inner sep=2]{};
      }
      \draw[black] (1) -- (2) -- (3) -- (4) -- (5) -- (6) -- (7) -- (8) -- (6);
      \draw[black] (3) -- (1);
      \draw[black] (2) -- (5) ;
      \draw[black] (4) -- (7) ;
      \draw[black] (1) to[out=0, in=0, distance = 3cm] (8) ;
      \node[above] at (1) {1};
      \node[right] at (3) {3};
      \node[left] at (2) {2};
      \node[right] at (4) {4};
      \node[left] at (5) {5};
      \node[left] at (6) {6};
      \node[below] at (8) {8};
      \node[right] at (7) {7};
    \end{tikzpicture}
    \caption{\texttt{4.16.a}}
  \end{subfigure}                                                
        \caption{Four trivalent graphs on eight vertices.
         Hyperplane sections
        of the corresponding graph curves in $\PP^4$ give
        self-dual configurations that realize
        four of the matroids in Table \ref{tab:rank4}.
                        \label{fig:graphs8}                                        }
\end{figure}

We now present the main result in this section: the classification of all
 simple self-dual matroids of  rank $5$. 
 Their realization spaces were represented algebraically as in
 Remark \ref{rmk:representveryaffine}.

\begin{theorem} \label{thm:1042} Up to isomorphism, there are $1042$
 simple self-dual matroids $M$ of rank~$5$. The dimension of their realization spaces is given in \eqref{eq:dimR(M)}.
Of these $1042$ self-dual matroids, precisely $346$ have $\mathcal{R}(M) = \emptyset$. 
At least three matroids $M$  have $\,\mathcal{R}(M) \not= \emptyset\,$ but $\,\mathcal{S}(M) = \emptyset$.
\end{theorem}

More comprehensive data, including Gr\"obner bases for  the ideals $I_{\mathcal{R}(M)}$
and $I_{\mathcal{S}(M)}$ and explicit
 self-dual realizations (whenever they exist) of the
  matroids $M$, can be found at \cite{MathRepo}.

\begin{proof}[Proof and discussion]
We computed a database of all rank $5$ self-dual matroids using {\tt OSCAR} in {\tt Julia} \cite{julia,oscar}.
Each matroid has the ground set $[10] = \{1, \dots, 10\}$.
For the construction of the matroids, we started from the complete list of rank $5$ matroids on $9$ elements from the database  {\tt polyDB} \cite{Paf}.
For each matroid in this list, we took the set of bases of the matroid, and closed this set under complements, considering each basis as a subset of $[10]$. If the resulting set of bases determined a matroid, and that matroid was not already isomorphic to a matroid in our database, we added it to the database. 
 This resulted in $1042$  matroids.
 
 We found the numbers  of nonbases for the $1042$ self-dual matroids to be as follows:
\begin{equation}
\label{eq:rank5nonbases}  \begin{matrix}
0^1, 2^1,\, 4^1, \, 6^3, \,8^7, \,10^{21},\, 12^{62}, \,14^{140}, \, 16^{208}, 
18^{182}, \, 20^{77}, \,22^{47},\, 24^{70},\, 26^{57}, \,\\ 28^{41}, \,
 \,30^{20}, \, 32^{20},\, 34^9, \, 
36^{14},\, 38^6,\, 40^7,\, 42^6,\, 44^9,\, 46^5,\, 48^6,\, 50^1,\, 
 52^1, \, \\ 54^1, \, 60^ 3,\, 62^1,\, 
64^1,\, 66^3,\, 68^2,\, 70^1,\, 72^3,\, 78^1,\, 84^1,\, 90^2,\, 132^1.
\end{matrix}
\end{equation}
The exponents indicate the number of matroids with that number of nonbases.

To find the realization spaces, we applied the  method from the proof
of Proposition~\ref{prop:rank4}. For all but $9$ of the matroids, we  succeeded in computing
a Gr\"obner basis for the ideal $I_{\mathcal{R}(M)}$ 
of the variety $\overline{\mathcal{R}(M)}$. This was done with
{\tt Magma} code parallelized with \texttt{GNU Parallel}~\cite{GNUparallel} on a server with four 12-Core Intel Xeon E7-8857 v2 at 3.0 GHz
and 1024 GB RAM. For the remaining $9$, we heuristically determined the dimension 
of the realization space $\mathcal{R}(M)$.

Because the Gr\"{o}bner basis computation is bigger for the rank $5$ case, we add one more optimization step compared to rank $4$. This step is possible for all but two matroids.
To compute the ideal  $I_{\mathcal{R}(M)}$ 
 for a matroid $M$, we construct a $5 \times 10$ matrix
$ X_M := (\,\text{Id}_5 \,| \, x_{ij} \,) $ by relabeling $M$. This choice of relabeling impacts the runtime of the calculations; we now describe our choice.
By a \emph{frame} for $M$, we mean a choice of the following three things:
\begin{enumerate}
  \item the matrix $X_M$, given by a choice of basis $b$ of $M$ and its complement. We  label the first five columns by the ordered basis  $b$ and the last five by the ordered complement.
  \item a column $j_0$ such that, under this labeling, we have $x_{ij_0} \neq 0$ for all $i$.
  \item nonzero elements $x_{j_1, 1}, \dots, x_{j_5, 5}$ (this scales each point projectively).
\end{enumerate} 
Each nonbasis of $M$ corresponds to a minor of $X_M$, by taking the columns in the nonbasis.
We start with the ideal $I'$ of
 nonbasis minors under the above choice of labeling, augmented by $x_{ij_0} -1$ and $x_{j_i, i} -1 $ for $1 \leq i \leq 5$. 
To make the comptutations feasible, we looped over all choices of frames for $M$ to minimize the product of the degrees of the 
generators of $I'$.  We then applied the methods in the proof of Proposition~\ref{prop:rank4}
for computing $I_{\mathcal{R}(M)}$ and $I_{\mathcal{S}(M)}$.

For the nine cases where we cannot compute 
the ideal $I_{\mathcal{R}(M)}$,
 we applied a heuristic to certify that the
very affine variety $\mathcal{R}(M)$ is non-empty, and to  determine its dimension.
 Namely, for various choice of $d$, we added $d$ random linear equations to $I'$.
After saturating, we expect to get a closed subvariety of codimension $d$
inside the realization space $\mathcal{R}(M)$. Using this method, we found 
$\,{\rm dim}(\mathcal{R}(M)) = 4\,$ for each of the
nine challenging matroids $M$.

In conclusion,
the dimensions of  $\mathcal{R}(M)$ for the $696$ realizable
self-dual matroids $M$~are 
\begin{equation}
\label{eq:dimR(M)}
 0^{16}, \,1^{113},\, 2^{220}, \,3^{175}, \, 4^{83},\, 5^{27},\, 6^{34}, \, 7^4,\, 8^{13},\, 9^1, \,10^6,\, 12^2,\, 14^1, \,16^1.
 \end{equation}
 Our {\tt Magma} implementation to compute the self-dual realization space was fully successful for $407$ of the $696$ realizable
self-dual matroids. 
In this data, we discovered three realizable~matroids that do not
admit a self-dual realization. One is presented in the next example.
\end{proof}

\begin{example} \label{ex:interesting}
We present a rank $5$ self-dual matroid $M$ with $\mathcal{R}(M) \not= \emptyset$
and $\mathcal{S}(M) = \emptyset$. Our matroid has $18$ nonbases. The eight nonbases
not containing the element $10$ are
$$ 12347,\, 12369,\, 12468,\, 12789,\, 13568,\, 14569,\, 23459,\, 23567,\, 36789. $$
The realization space $\mathcal{R}(M)$ is a very affine threefold. Its points are the configurations
$$
X \,\,\, = \,\,\, \begin{small}
\begin{pmatrix}
1 & 0 & 0 & 0 & 0 & 1 & 1 & 1 & 0 & 1 \\
 0 & 1 & 0 & 0 & 0 & 1 & a & b & 1 & 1 \\
 0 & 0 & 1 & 0 & 0 & 1 & c & d & 1 & 0 \\
 0 & 0 & 0 & 1 & 0 & 1 & 1 & b & e & f \\
 0 & 0 & 0 & 0 & 1 & 1 & 0 & d & e & g
 \end{pmatrix}, \end{small}
 $$
 where the $236$ basis minors are nonzero, and the $16$ nonbasis minors impose the equations
  $$ \!
ce{+}g \,=\, a e {-} e {+} f{-} 1 \,=\, a d {-} b c{+}c{-}d \, = \,
  b g {+} d e {-} d g {-} d \,=\, b e {+} b f {+} d e {-} d g {-} 2 b {-} d{-}e {+} g{+}1\, = \, 0.
  $$
  We now introduce $10$ new variables, by setting  $\Lambda = {\rm diag}(\lambda_1,\lambda_2,\ldots,\lambda_{10})$.
 Consider  the ideal  in $\QQ[a,b,\ldots, g, \lambda_1,\ldots,\lambda_{10}]$ that is
  generated by the entries of the $5 \times 5$ matrix     $\,X \cdot \Lambda \cdot X^T $.
  One checks that   the monomial  $\lambda_2 \lambda_5$ lies in that ideal. From this
  we conclude that    $\,\mathcal{S}(M) = \emptyset$.  
\end{example}

\begin{remark} By Proposition \ref{prop:stable}, only
 one of the $1042$ self-dual matroids in Theorem \ref{thm:1042}
is not stable. That matroid  has $120 =\binom{4}{2} \cdot \binom{6}{3} $ bases,
and it is the direct sum of the  two uniform matroids, of rank $2$ on $[4]$
and of rank $3$ on $[6]$. Geometrically, this is the special configuration 
of ten points in $\PP^4$  discussed
in \cite[Remark 3, page 47]{DO}, namely
four collinear points plus six coplanar points.
The realization space $\mathcal{R}(M) $ is $5$-dimensional.
Its subvariety $\mathcal{S}(M)$ has dimension $4$, and it 
arises by requiring that the six coplanar points lie on a conic.
\end{remark}

From the computations for Theorem \ref{thm:1042},
we obtained explicit sample points in $\mathcal{S}(M)$
for all but $6$ of the $692$ non-uniform matroids $M$ for which this very affine variety has not been shown to be empty. In the cases where a Gr\"obner basis
was reached, we often also succeeded in finding sample points with coordinates in $\QQ$. In the other cases,
we constructed points with coordinates in a number field. 
Equipped with these data, we return to the question that was answered for rank $4$ in Proposition \ref{prop:9of12}:
which of the  rank $5$ matroids $M$ that admit self-dual realizations arise as hyperplane sections of
genus $6$ canonical curves? This question will be addressed in the next section, which 
concerns curves whose genus is between $6$ and~$10$.

\section{From genus six to Mukai Grassmannians}
\label{sec4}

Given any stable self-dual configuration $X$, it is natural to ask 
whether $X$ can be obtained as a hyperplane
section of a canonical curve $C$ in $\PP^{g-1}$. We explored this  for $g \leq 5$. Table~\ref{tab:genus5}
shows a list of smooth canonical curves of genus $5$ for
nine matroids.
 We now ask the same question when the genus is 
in the range $ 6 \leq g \leq 10$. 
This range allows for a representation of 
 self-dual configurations
 as linear sections of a  higher-dimensional variety, here called the
   {\em Mukai Grassmannian}~\cite{Kap2, Mukai, Mukai2}.
This theory is well-known to experts on canonical curves and K3 surfaces.
The purpose of this section is to furnish an exposition that emphasizes
computations and combinatorics.
 For instance, we list ideal generators for the 
Mukai Grassmannians explicitly, following~\cite{Kap1}.
Our main contribution contains the lifting algorithm for self-dual points in linear general position to genus six canonical curves translated from \cite{Bath}. An implementation together with extensive computations supports our conjecture that this algorithm holds more generally.
For $g = 7,8,9$, the lifting problem was studied by Petrakiev \cite{Pet}, 
and we present his results from our perspective.
This section is our report on 
first steps towards  being able to transition {\em in practice}
between  canonical curves and self-dual matroids.

We begin with the  first non-trivial case,  $g = n+1=6$.
We shall present an algorithm
for realizing a self-dual configuration $X$ of $10$ points in $\PP^4$ as a hyperplane
section of a canonical curve $C$ in $\PP^5$.
The input $X$ will be required to satisfy a 
genericity assumption, and we aim for a curve
$C$ that is smooth. Our algorithm was implemented in
{\tt Macaulay2} \cite{M2}.
We applied it to
a wide range of configurations arising from the matroids in
Theorem~\ref{thm:1042}.

The algorithm follows the  discussion in the 1938 paper by  Bath \cite{Bath}.
A modern interpretation was given by Eisenbud and Popescu in
 \cite[Remark 9.1]{EP}.
Our point of departure is the following 
well-known fact about the equations that define a
genus $6$ canonical curve~$C$.
We are interested in the homogeneous prime ideal
$I_C \subset \CC[x_0,x_1,\ldots,x_5]$ and syzygies thereof.

\begin{lemma} \label{prop:skewsymm}
Let $C$ be a genus $6$ canonical curve that is neither hyperelliptic, nor trigonal, nor a plane quintic.
The ideal $I_C$ of  the curve $\,C$ is generated by six quadrics, where
five of the quadrics can be chosen as the $4 \times 4$-subpfaffians of a skew-symmetric
$5 \times 5$ matrix of linear forms. The first and second syzygies are summarized in the Betti table
of $I_C$, which~is
\begin{equation} \label{eq:betti6}
\begin{bmatrix} 6 & 5 & . \\ . & 5 & 6 \end{bmatrix}.
\end{equation}
\end{lemma}

The symmetry in the Betti table (\ref{eq:betti6}) is the
Gorenstein property of $C$, which is preserved
under hyperplane sections. Indeed, as emphasized in \cite{EP, EP2},
the Gorenstein property is a key feature of self-dual configurations.
For our algorithm we assume that $X$ is a self-dual configuration
in $\PP^4$ whose homogeneous radical ideal $I_X$ has the Betti table (\ref{eq:betti6}).

\begin{algo}[Lifting self-dual configurations of $10$ points to genus $6$ canonical curves] \label{alg:genus6}

\underline{{\rm Input}}: The ideal $I_X$ in $\QQ[x_0,\ldots,x_4]$ of a self-dual configuration in $\PP^4$
with Betti table (\ref{eq:betti6}). \smallskip
\\
\underline{{\rm Output}}: The ideal $I_C$ in $\QQ[x_0,\ldots,x_4,x_5]$ of a canonical curve $C$
which  satisfies $I_C|_{x_5 = 0}= I_X$. 
\begin{enumerate}
\item Pick four general quadrics from the ideal $I_X$ and let $J$ be the ideal they generate. \vspace{-0.2cm}
\item Compute the ideal quotient $K =J:I_X$. 
This gives six points in a hyperplane in $\PP^4$. \vspace{-0.2cm}
\item The ideal $K$ is generated by one linear form $\ell$ and four quadrics. Working modulo $\ell$,
write the quadrics in only four variables and compute two linear syzygies they satisfy.  \vspace{-0.2cm}
\item Multiply this $2 \times 4$ syzygy matrix with a random $4 \times 3$ matrix over $\QQ$.
This gives a $2 \times 3$ matrix whose entries are linear forms.
Its $2 \times 2$ minors define a twisted cubic curve.  \vspace{-0.2cm}
\item Let $L$ be the ideal generated by $\ell$ and these minors. The intersection $L \cap I_X$
is generated by three quadrics and one cubic. Let $M$ be the ideal generated by the three quadrics. \vspace{-0.2cm}
\item The ideal $M$ defines a curve of degree $8$ in $\PP^4$. The twisted cubic
is one component. The other component is an elliptic normal curve, given by
 the ideal quotient $N = M:L$. \vspace{-0.2cm}
\item
The ideal $N$ is generated by five quadrics. Compute their module of first syzygies.
This module is generated by five linear syzygies.
By changing bases if necessary, write them
 as a skew-symmetric $5 \times 5$ matrix $\Sigma$  whose entries
are linear forms in $\QQ[x_0,\ldots,x_4]$. \vspace{-0.2cm}
\item The ideal ${\rm pf}_4(\Sigma)$ generated by the
five $4 \times 4$ subpfaffians  of $\,\Sigma$ 
is contained in $I_X$. At least one of
the six generators of $I_X$ does not lie in ${\rm pf}_4(\Sigma)$.
Let $q$ be such a quadric. \vspace{-0.2cm}
\item Output the representation $\,I_X = \langle q \rangle + {\rm pf}_4(\Sigma)\,$ for the ideal of the ten given points.
\vspace{-0.2cm}
\item Lift $q$ to a quadric $\tilde q$ in six variables
by adding $x_5$ times a random linear form.
Let $\tilde \Sigma$ be the sum of $\,\Sigma$ 
plus $x_5$ times a random skew-symmetric matrix with entries in $\QQ$. \vspace{-0.2cm}
\item Set $\,I_C = \langle \tilde q \rangle + {\rm pf}_4(\tilde \Sigma)\,$ to be the representation  for the ideal of the canonical curve $C$.\vspace{-0.2cm}
\item Verify $C$ is non-singular: if so, output $I_C$.
If not, restart the algorithm at step (1).

\end{enumerate}
\end{algo}
\begin{conjecture}
 \label{conj:alggenus}
Algorithm \ref{alg:genus6} solves
the lifting problem for any self-dual configuration $X$ of $10$ points in $\PP^4$ whose ideal $I_X \subset \QQ[x_0, \dots, x_4]$ has the Betti table \eqref{eq:betti6}. \end{conjecture}

Bath's construction in \cite[Section~2]{Bath} (cf. \cite[Theorem~9.10]{EP}) proves this conjecture for configurations $X$ in
linearly general position, i.e.~when the self-dual matroid of $X$ is uniform. The point of
Conjecture~\ref{conj:alggenus} is that the
 matroid of $X$ need not be uniform.
 We verified the conjecture experimentally. We applied Algorithm \ref{alg:genus6} to many special
 configurations, arising from self-dual realizations of matroids in the database for rank $5$ in Theorem \ref{thm:1042}.
In Section \ref{sec3}, we described our computation of rank $5$ self-dual matroids $M$. 
For the $692$ non-uniform matroids for which $\mathcal{S}(M) \not= \emptyset$ 
has not been ruled out, we produced self-dual configurations realizing $686$ of the matroids. Of these, $283$ configurations had all $10$ points defined over $\QQ$. For these $283$ configurations, $269$ had the Betti table \eqref{eq:betti6}. For each realization of these $269$ matroids we considered, our {\tt Macaulay2} implementation of
Algorithm~\ref{alg:genus6} provides a smooth genus $6$ curves of degree $10$ with the expected Betti table passing through the points.

Algorithm~\ref{alg:genus6} also extends readily to points defined over a number field $K/\QQ$. However, depending on the degree of the number field and height of the coordinates, our implementation of the algorithm can take between a few minutes and several hours
to  terminate. Of the remaining $403$ self-dual configurations defined over number fields, $402$ have Betti table \eqref{eq:betti6}, and for $217$ of these, we produced smooth genus $6$ curves passing through the points.

\begin{example}[Petersen Graph]
In Example \ref{ex:petersen}  we study the point configuration 
 (\ref{eq:petersenmatrix}). This is a 
 hyperplane section of a reducible genus $6$ canonical curve whose dual graph is the Petersen graph in Figure \ref{fig:ptersengraph}.
 The ideal of this configuration has the Betti table~\eqref{eq:betti6}. Algorithm~\ref{alg:genus6}
  computes a smooth  genus $6$ canonical curve with Betti table \eqref{eq:betti6} passing through this configuration.
 The curve is given by the  $4 \times 4$-subpfaffians of the skew-symmetric
 $5  \times 5$~matrix 
 $$ 
\begin{small} \!\!
 \frac{31}{25} \! \begin{pmatrix}
0& x_3&\!\! \frac{-1}{9}(5x_2{+}x_3)\!\!& \frac{1}{5}x_3& x_1+ \frac{6}{5}x_3 \\
-x_3& 0& \frac{2}{3}x_3& -x_0-2x_3& -x_1{-}x_3{+}x_4 \!\! \\
 \! \! \frac{1}{9}(5x_2{+}x_3)& \frac{-2}{3}x_3& 0&\!\!\!\! -x_1{-}\frac{1}{45}x_3{-}\frac{1}{3}x_4 \!\! & \!\! \frac{-1}{9} (\frac{31}{5}x_3+x_4) \!\!\! \\
 \frac{-1}{5}x_3& x_0+2x_3& \!\! \!\! x_1{+}\frac{1}{45}x_3{+}\frac{1}{3}x_4 \!\! & 0& \!\! 2x_3{-}x_2{-}\frac{7}{5}x_4\\
 \! -x_1{-}\frac{6}{5}x_3&\! \! x_1{+}x_3{-}x_4 \!& \!\! \frac{1}{9} (\frac{31}{5}x_3+x_4)& {-}2x_3{+}x_2{+}\frac{7}{5}x_4& 0 
        \end{pmatrix} +
        \begin{pmatrix}
0 & \!\! 0 & \!\! 0 & \!\!\! \! \! -x_5 \! &\!\!\!\! - x_5\!\! \\
0 &\!\! 0 & \!\! 0 &\!\! x_5 &\!\! 0 \!\\
0 & \!\! 0 & \!\! 0 & \!\!\ 0 & \!\!x_5 \!\\
x_5 & \!\!\! -x_5 \!\! & \!\! 0 &\!\! 0 & \!\!\!\!\!\! -x_5 \!\! \! \\
x_5 & \!\! 0 &\!\!\!\! \! -x_5 \! &\!\! x_5 &\!\!  0 
        \end{pmatrix}
        \end{small}
$$
along with the  quadric
$\,6x_0x_3+9x_1x_3+5x_2x_3+x_3^2+8x_3x_4+x_0x_5 $.
Set $x_5=0$ to get (\ref{eq:petersenmatrix}).
\end{example}

We now turn to canonical curves of genus $g =7,8,9,10$.
These curves have a beautiful representation,
due to Mukai \cite{Mukai, Mukai4, Mukai3, Mukai2},
 as linear sections
of a certain fixed variety $G$ in a higher-dimensional
projective space $\PP^{N}$.
We refer to $G$ as the {\em Mukai Grassmannian} in genus~$g$.
Setting $D = {\rm dim}(G)$, the representation works as follows.
Consider a general linear subspace $L$ of dimension $\ell$ in $\PP^N$.
The intersection $L \cap G$ is a variety of dimension~$D + \ell - N$.  

For an interesting  range of dimensions $\ell$, the following geometries appear:
\begin{itemize}
\item[(a)] If $D + \ell - N =3$, then $L \cap G$ is a Fano threefold in $L$. \vspace{-0.3cm}
\item[(b)] If $D+ \ell - N = 2$,  then $L \cap G$ is a K3 surface in $L$. \vspace{-0.3cm}
\item[(c)] If $D+ \ell -N = 1$,  then $L \cap G$ is a canonical curve in $L$. \vspace{-0.3cm}
\item[(d)] If $D + \ell-N = 0$, then $L \cap G$ is a self-dual configuration in $L$.
\end{itemize}
If $g \in \{7,8,9\}$, then the representation (c) is universal \cite{Mukai4, Mukai3, Mukai2}:
every general canonical curve $C$ of genus $g$ satisfies $C = L \cap G$
for some subspace $L$. For $g=10$, this representation imposes
a codimension one condition on the curve $C$.
These facts are well-known in algebraic geometry.
 The connection to K3 surfaces in (b) and
Fano threefolds in (a) has been a central element of Mukai's theory.
We focus on the  representation of self-dual configurations in (d).

The Mukai Grassmannians $G$ suggest  two computational problems.
The {\em forward problem} asks: given a linear space $L$,
compute the resulting self-dual configuration $X = L \cap G$.
The {\em lifting problem} asks: given a  self-dual configuration $X$ in $\PP^{g-2}$,
compute a linear space $L$ such that $X = L \cap G$. 
The lifting problem is much harder than the forward problem.

Before examining these two problems, we shall review the Mukai Grassmannians.
We follow the presentation by Kapustka \cite{Kap1, Kap2}.
 We furnish code in {\tt Macaulay2}
that creates a prime ideal {\tt G} in a polynomial ring. 
The variety of {\tt G} is the Mukai Grassmannian $G$ in $\PP^N$.
The Betti table for each 
geometry agrees with that of the Mukai Grassmannian.

\bigskip \noindent \underbar{$g=7$}: 
Here, $N = 15$, $D = 10$, $\,C$ is a curve of degree $12$ in $\PP^6$, and
$X$ is a self-dual configuration of $12$  points in $\PP^5$.
The Mukai Grassmannian $G$ is the {\em orthogonal
Grassmannian}  \cite[(3.1)]{Kap1}, \cite[Section 3.2]{Kap2}.
It is parametrized by the $16 = 1 + 10 + 5$
subpfaffians of a skew-symmetric $5 \times 5$-matrix.
They  satisfy $10$ quadratic relations, and these 
generate the ideal:
\begin{verbatim}
R = QQ[a..p]; G = ideal(k*m-j*n+h*o-e*p, k*l-i*n+g*o-d*p, j*l-i*m+f*o-c*p, 
       h*l-g*m+f*n-b*p, e*l-d*m+c*n-b*o, h*i-g*j+f*k-a*p, e*i-d*j+c*k-a*o, 
                        e*g-d*h+b*k-a*n, e*f-c*h+b*j-a*m, d*f-c*g+b*i-a*l );
\end{verbatim}
\vspace{-0.1in}
The ideal {\tt G} is Gorenstein. Its resolution has the Betti table
\begin{tiny} $ \,\,\begin{bmatrix} 10 & 16 & . & . \\ . & . &  16 & 10 \end{bmatrix}$. \end{tiny}

\medskip \noindent \underbar{$g=8$}:
Here, $N=14$, $D=8$, $C$ is a curve of degree $14$ in $\PP^7$, and
$X$ is a self-dual configuration of $14$ points in $\PP^6$.
Here {\tt G} is the Pl\"ucker ideal of  the classical Grassmannian ${\rm Gr}(2,6) \subset \PP^{14}$,
i.e.~the $4 \times 4$-subpfaffians of a skew-symmetric $6 \times 6$-matrix \cite[Section 3.1]{Kap1}:
\begin{verbatim}
R = QQ[a..o]; G = Grassmannian(1,5,R);
\end{verbatim}
\vspace{-0.1in}
The ideal {\tt G} is Gorenstein. Its resolution has the Betti table
\begin{tiny} $ \,\,\begin{bmatrix} 15 & 35 & 21 &  . & . \\
 . & . &  21 & 35 & 15 \end{bmatrix}$. \end{tiny}

\medskip \noindent \underbar{$g=9$}:
Here, $N=13$, $D=6$,  $\,C$ is a curve of degree $16$ in $\PP^8$,
 $X$ consists of $16$ points in $\PP^7$, and
 $G$ is the {\em Lagrangian 
Grassmannian} ${\rm LGr}(3,6)$. Its ideal {\tt G} is generated by
$21$ quadrics \cite[Section 3.3]{Kap1}, \cite[(3.2)]{Kap2}. 
These are the {\em $12$ edge trinomials} and {\em $9$ square trinomials} in \cite{Boege}.
 \begin{small}
\begin{verbatim}
R = QQ[a..n]; G = ideal(f*k-e*l+j*n,e*k+f*m-i*n,f*j-b*k-d*l,e*j-d*k+b*m-a*n,
  f*i+d*k-c*l+a*n,e*i-c*k-d*m,g*h+2*d*k-c*l-b*m+a*n,f*h-j*k+i*l,k^2+l*m+h*n,
    e*h+i*k+j*m,d*h+i*j-a*k,c*h+i^2+a*m,b*h-j^2+a*l,f^2+g*l+b*n,e*f+g*k-d*n,
    a*f-b*i-d*j,e^2-g*m-c*n,d*e-c*f-g*i,b*e+d*f+g*j,a*e+d*i-c*j,b*c+d^2+a*g);
\end{verbatim}
\end{small}
\vspace{-0.07in}
The ideal {\tt G} is Gorenstein. Its resolution has the Betti table
\begin{tiny} $ \begin{bmatrix} 21 \! & \! 64 \! & \! 70 \! & \! .\! & \!. \! & \! . \\
. \! & \! .\! &\! . \! & \! 70 \! & \! 64 \! & \! 21 \end{bmatrix}$. \end{tiny}

\medskip \noindent \underbar{$g=10$}:
Here, $N=13$, $D=5$,  $\,C$ is a curve of degree $18$ in $\PP^9$, and
$X$ consists of $18$ points in $\PP^9$. Moreover, $G$
 is an orbit closure of 
the adjoint representation of the {\em exceptional Lie group $G_2$}.
Its ideal {\tt G} has the following matrix representation,
given in \cite[Section 3.4]{Kap1}.
\begin{small}
\begin{verbatim}
R = QQ[a..n]; G = pfaffians(4, matrix {{0,-f,e,g,h,i,a},
 {f,0,-d,j,k,l,b},{-e,d,0,m,n,-g-k,c},{-g,-j,-m,0,c,-b,d},
 {-h,-k,-n,-c,0,a,e},{-i,-l,g+k,b,-a,0,f},{-a,-b,-c,-d,-e,-f,0}});
\end{verbatim}
\end{small}
The ideal {\tt G} is Gorenstein, and its Betti table equals
\begin{tiny} $\! \begin{bmatrix} 
    28 & 105  & 162 & 84  &    .    &    .   &    .   \\
    .     &    .   &    .   & 84  & 162 & 105 & 28  \\
\end{bmatrix}$. \end{tiny}
              
 Equipped with explicit quadrics for the four Mukai Grassmannians $G$,
one can experiment with the forward problem: pick a (random)
subspace $L$ and examine the
(matroid of the) self-dual configuration $X = L \cap G$.
Here we want to fix an isomorphism $L \simeq \PP^{g-2}$,
so that $X$ is given by a $(g-1) \times (2g-2)$ matrix, 
as in the previous sections. An intermediate step is
to represent $X$ by its
ideal $I_X$ in the polynomial ring
$\QQ[x_0,x_1,\ldots,x_{g-2}]$.
After creating {\tt G} with one of the  four code fragments above,
this ideal is now found in {\tt Macaulay2} as follows:
\begin{small}
\begin{verbatim}
  g = 10; S = QQ[x_0..x_(g-2)]; L = map(S,R,apply(# gens R,t->random(1,S)));
  IX = L(G);  dim IX, degree IX
\end{verbatim}
\end{small}

This process provides samples from the space
$\mathcal{S}(n)$, for $n=g-1$. These
can be compared to the samples obtained from 
Algorithm \ref{alg:SOn}. The new method has pros and cons.
On the one hand, $X$ now lies on many canonical curves:
simply take $C = L' \cap G$, where $L' \simeq \PP^{g-1}$
is any subspace that contains $L$.
On the other hand, to write down the 
$n \times 2n$ matrix $X$, 
we must solve the equations in {\tt IX},
either numerically or symbolically.
One worthwhile experiment is to
record  the number of real points in $X = L \cap G$,
for a large sample of subspaces $L$.
In light of the next result,
we expect that this number can
be any even integer between $0$ and $2n$.

\begin{proposition}
For every $n \geq 4$, there exists a self-dual configuration $X$
of $2n$ points in $\PP^{n-1}$, lying on a real canonical curve $\,C$ ,
where all coordinates of all points are real numbers.
\end{proposition}

\begin{proof} We fix an M-curve of genus $g=n+1$.
This is a smooth curve $C$, defined over $\RR$, whose real locus has $n+2$ connected components. 
Holomorphic differentials have an even number of zeros on each connected
component \cite[Corollary 4.2.2]{CP}.
Hence every component of $C$ is an oval and not a pseudoline. We select $n$ of the ovals and we fix
one point on each of them. The $n$ points span a hyperplane in $\PP^{n-1}$.
That hyperplane
intersects each oval twice, so
it meets the curve in $2n$ real points.
This is our fully real configuration $X$ in $\PP^{n-1}$.
\end{proof}
   
\begin{remark}
There are also {\tt Macaulay2} packages for 
 random canonical curves and K3~surfaces.
 The state of the art for curves is
the package {\tt RandomCanonicalCurves}
by von Bothmer and Schreyer which works up to $g=14$.
Hoff and Staglian\` o \cite{HS} developed the package
{\tt K3s} which created embeddings of K3 surfaces.
By slicing the resulting curves and surfaces, we  can sample from
self-dual configurations over finite fields and explore their matroids.
\end{remark}

We now come to the lifting problem. The input is a self-dual configuration $X$ in $\PP^{n-1}$.
We can ask for a lifting of $X$
 to the Mukai Grassmannian, or just to a
canonical curve. To discuss the latter, we write
 $\mathcal{K}(n)$ for the closure of the set of self-dual configurations $X$
 in $\mathcal{S}(n)$
that arise as hyperplane sections
 of some smooth canonical curve
in $\PP^n$.
We call such $X$ {\em canonical configurations}.
The ideals {\tt G} above allow us to
sample from  $\mathcal{K}(n)$
when $n \leq 9$.

\begin{theorem}[Petrakiev \cite{Pet}] \label{thm:petr}
The variety $\mathcal{K}(n)$ of canonical configurations is irreducible.
It satisfies $\,\mathcal{K}(n) = \mathcal{S}(n)\,$ for $n \leq 7$,
but the strict inclusion $\,\mathcal{K}(n) \subsetneq \mathcal{S}(n)\,$ holds for $n = 8$.
\end{theorem}

The irreducibility of $\mathcal{K}(n)$ follows from the fact that
the moduli space $\mathcal{M}_g$ is irreducible. 
For $2\leq n\leq 5$ the second statement follows from observations in Sections \ref{sec2} and \ref{sec3}.
For $n=6,7 $ and $8$, the second assertion
is based on a dimension argument. In these cases,
the dimension of $\mathcal{S}(n)$ equals $15,21$ and $28$, by Theorem \ref{thm:SOn}.
For the dimension of $\mathcal{K}(n)$, one considers the Mukai Grassmannian
and its symmetries, and one looks at linear sections modulo these symmetries.
Using such an approach, Petrakiev \cite{Pet} shows that
${\rm dim}(\mathcal{K}(6)) =  15$ and ${\rm dim}(\mathcal{K}(7)) = 21$, but
 ${\rm dim}(\mathcal{K}(8)) \leq 27$. 
  He conjectures  ${\rm dim}(\mathcal{K}(8)) = 27$.
  Whenever the dimensions agree, we have
    $\,\mathcal{K}(n) = \mathcal{S}(n)\,$ 
  because both varieties are irreducible.

One might hope to
extend Algorithm \ref{alg:genus6}
from $n=5$ to $n=6,7$. This makes sense because
$\mathcal{S}(n) = \mathcal{K}(n)$ for $n=6,7$.
Indeed, by \cite[Theorem 2.6]{Pet},
every general self-dual configuration
$X$ of $12$ points in  $\PP^5$ can be
written as $X = L \cap G$,
where $G$ is the spinor variety in $\PP^{15}$
and $L \simeq \PP^5$ is a subspace of $\PP^{15}$.
The analogous statement holds for $14$ points in $\PP^6$.
However, a syzygy-based method like Algorithm~\ref{alg:genus6} to find $L$ is currently not known.

The first-principles approach to the lifting problem is to encode it as
a polynomial system.
We shall describe these equations for $n=7$.
Our input is a self-dual configuration $X$ of $14$ points in $\PP^6$.
The task is to write $X$ as a linear section of the Grassmannian $G ={\rm Gr}(2,6)$.
Points on $G$ are skew-symmetric 
$6 \times 6$-matrices $P = (p_{ij})$ of rank two.  Our ansatz for each 
 matrix entry is a linear form $p_{ij} = \sum_{k=0}^6 c_{ijk} x_k$.
The $c_{ijk}$ are unknowns, so we have $15 \times 7 = 105$ unknowns in total.
Consider the $15 = \binom{6}{4}$ subpfaffians of size $4 \times 4$.
We require that these vanish at the $14$ points in $X$.
This gives a system of $210$ quadratic equations in the $105$ unknowns.
One could try to solve this large system of equations numerically.
Every floating point solution $(c_{ijk})$ would represent a
subspace $L \subset \PP^{14}$ such that $X = L \cap  G$.

Some complexity reduction can be achieved by
symmetries. The pfaffian of $P$ is
a cubic hypersurface in $\PP^6$ which is
singular at the $14$ points in $X$.
Our task is to find a Pfaffian cubic with 
prescribed singularities. Now, the hypersurface
is left unchanged when $P$ is replaced by
$U P U^T$ for some $U \in {\rm GL}(6)$.
Using a normal form for this group action,
we can reduce from $105$ unknowns to $ 69$ unknowns.
Still, it is a formidable challenge to
solve the equations.

In this section we discussed linear sections $X$ of canonical curves.
We focused on the regime where the
curves are linear sections of a
 Mukai Grassmannian. Of course, the $X$
  are also linear sections of the
Fano threefolds and K3 surfaces in bullets (a) and (b) above.
In our combinatorial setting, any of these
varieties can be replaced by sufficiently nice degenerations.
We note that each Mukai Grassmannian admits
toric degenerations $G \leadsto  {\rm in}(G) $ in $\PP^N$.
By slicing the resulting  toric variety ${\rm in}(G)$, we  also obtain
self-dual configurations $X = L \cap {\rm in}(G)$ that lie in $\mathcal{K}(n)$.
For one concrete example of a toric degeneration in the $n=8$ case see
 \cite[Figure 3]{Boege}.
Which self-dual matroids of rank $8$
arise from that $6$-dimensional toric variety?

One important degeneration of K3 surfaces needs to be
mentioned here. This is the {\em tangent developable surface} $S_g$
which gained prominence in the recent advances on 
canonical curves described  in \cite{EL}.
Following \cite[Equation (2.2)]{EL}, 
$S_g$ is the surface in $\PP^g$ that is parametrized by
a vector $(u,v)$ times the Jacobian matrix
of $(s^g, s^{g-1} t, s^{g-2} t^2, \ldots, t^g)$.
Hyperplane sections of $S_g$ are rational curves with $g$ cusps
that behave like smooth canonical curves as far
as syzygies go. Therefore, by taking subspaces $L$ of codimension $2$ in $\PP^g$,
we obtain self-dual configurations $X = L \cap S_g$.
Which self-dual matroids arise from such $X$?

\section{Graph curves}
\label{sec5}

In the previous section, we saw that canonical models of
general smooth curves, of moderate genus $g$, are
linear slices of Mukai Grassmannians. In what follows we
turn to the other end of the spectrum, namely we shall study an important
class of reducible curves with many components. These are 
graph curves, which were introduced  by Bayer and Eisenbud in \cite{BE}. 

A {\em graph curve} $C_G$ is a stable canonical curve that consists of
$2g-2$ straight lines in $\PP^{g-1}$ such that every
line intersects precisely three other lines. The dual graph $G$ of the curve
$C_G$ is trivalent, so it has $2g-2$ vertices and $3g-3$ edges, and it is 
$3$-connected~\cite[Proposition 2.5(ii)]{BE}. The graph $G$
determines the abstract curve $C_G$. It
specifies a unique point in the Deligne--Mumford
compactification $\overline{\mathcal{M}}_g$ of the 
$(3g-3)$-dimensional moduli space $\mathcal{M}_g$
of smooth curves of genus $g$. In tropical
geometry \cite{Mel}, one assigns lengths to the edges of $G$, and one views these lengths
as local coordinates on the combinatorial space 
${\rm trop}(\mathcal{M}_g)$.

\begin{example}[$g=5$] \label{eq:cubecurve} Let $G$ be the edge graph of
  a $3$-cube. It has $8$ vertices and $12$ edges.
  The graph curve $C_G$ consists of eight lines in $\PP^4$,
  and it is obtained by intersecting three reducible quadrics, each of
  which is the union of two general hyperplanes.
  The choice of six  general hyperplanes in $\PP^4$ is unique, 
  up to projective transformation, and hence so is $C_G$.
\end{example}

This last observation about the projective uniqueness
of the curve holds more generally.
 
\begin{proposition} \label{prop:graphcurveunique}
  Fix any trivalent $3$-connected graph $G$ with $2g-2$ vertices,
  and let ${\rm Cyc}_G$ be any $\,g \times (3g-3)$ matrix whose
  rows span the cycle space of $G$. This matrix has $2g-2$
  linearly dependent triples of columns, and these triples
  span the $2g-2$ lines that make up the graph curve $C_G$. The curve
  $C_G$ is unique up to projective transformations in $\PP^{g-1}$.
\end{proposition}

\begin{proof}
  The row space of ${\rm Cyc}_G$ describes a realization of the graphic matroid of $G$. 
    Its dual is the {\em bond matroid} of $G$.
  The columns of
  ${\rm Cyc}_G$ realize the bond matroid of $G$ in $\PP^{g-1}$.  Since both matroids are regular, 
  the realizations given by ${\rm Cyc}_G$ are projectively unique \cite{BL76}. 
  The columns of ${\rm Cyc}_G$ realize the intersection points of the lines in $C_G$. Since
   $G$ is trivalent and $3$-connected, the three edges incident to the same vertex in $G$ form a $3$-cycle in the bond matroid.  So we can build a unique line for each vertex of $G$, and these lines intersect as determined by $G$.
  This proves the claim.  
  See \cite[Theorem 8.1]{BE} for a more algebraic proof.
\end{proof}

\begin{example}[$g=6$] \label{ex:petersenpart1} Let $G$ be the
  {\em Petersen graph}, with vertex set $\{0,1,2,3,4,5,6,7,8,9\}$,
  and edges directed as in Figure \ref{fig:ptersengraph}. 
  Six linearly independent  cycles form the rows of the matrix
  \setcounter{MaxMatrixCols}{20}
  $$ {\rm Cyc}_G \,\, = \,\, \begin{small} \begin{blockarray}{ccccccccccccccc} 12&23&34&45&51&68&80&07&79&96&16&27&38&49&50 \\
      \begin{block}{(ccccccccccccccc)}
        -1& -1& -1& -1& 1& \phantom{-}0& \phantom{-}0& 0& \phantom{-}0& 0& \phantom{-}0& 0& 0& 0& 0 \,\, \\
        \phantom{-} 0& \phantom{-}0& \phantom{-}0& \phantom{-}0& 0& -1& -1& 1& -1& 1& \phantom{-}0& 0& 0& 0 &0\,\,\\
        \phantom{-} 1& \phantom{-}0& \phantom{-}0& \phantom{-}0& 0& -1& -1& 1& \phantom{-}0& 0& -1& 1& 0& 0& 0 \,\,\\
        \phantom{-} 1& \phantom{-}1& \phantom{-}0& \phantom{-}0& 0& -1& \phantom{-}0& 0& \phantom{-}0& 0& -1& 0& 1& 0& 0 \,\,\\
        \phantom{-} 1& \phantom{-}1& \phantom{-}1& \phantom{-}0& 0& -1& -1& 1& -1& 0& -1& 0& 0& 1& 0 \,\,\\
        \phantom{-} 1& \phantom{-}1& \phantom{-}1& \phantom{-}1& 0& -1& -1& 0& \phantom{-}0& 0& -1& 0& 0& 0& 1 \,\,\\
      \end{block}
    \end{blockarray} \end{small}. \vspace{-0.3cm}
  $$
  The last row is the cycle $1{-}2{-}3{-}4{-}5{-}0{-}8{-}6{-}1$.
  The columns of ${\rm Cyc}_G$ give $15$ points in $\PP^5$. Each of the ten lines
  of $C_G\subset \PP^5$ passes through three of the $15$ points. The lines
  are labeled by the vertices $0,1,\ldots,9$ of $G$.
  For instance,
  the line  $L_1$ is spanned by the points $12,51,16$,
  and it satisfies the four linear equations
  $x_2 = 0$ and $x_3 = x_4 = x_5 = x_6$.
  Similarly, line $L_2$ is spanned by the points $12,23,27$,
  and its equations are $x_2=0$ and $-x_1 = x_4 = x_5 = x_6$.
\end{example} 

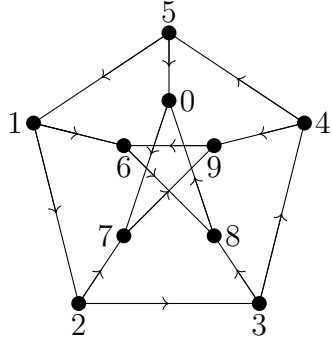
\begin{figure}[h]
    \centering
    \begin{tikzpicture}[scale = 1.2]
      \coordinate (2) at (0,0);
      \coordinate (3) at (2,0);
      \coordinate (4) at (2.5,2);
      \coordinate (5) at (1,3);
      \coordinate (1) at (-0.5,2);
      \coordinate (6) at (0.5,1.75);
      \coordinate (7) at (0.5,0.75);
      \coordinate (8) at (1.5,0.75);
      \coordinate (9) at (1.5,1.75);
      \coordinate (0) at (1,2.25);
      \foreach \i in {1,2,3,4,5,6,7,8,9,0} {
        \path (\i) node[circle, black, fill, inner sep=2]{};
      }
      \draw[black] (0)--(5) -- (4) -- (3) -- (2) -- (1) -- (6) -- (9) -- (7) -- (0) --(8) --(3);
      \draw[black] (9) -- (4) ;
      \draw[black]  (8)--(6) -- (1) --(5);
      \draw[black] (2) --(7);
      \node[left] at (1) {1};
      \node[below] at (3) {3};
      \node[below] at (2) {2};
      \node[right] at (4) {4};
      \node[above] at (5) {5};
      \node[below] at (6) {6};
      \node[right] at (8) {8};
      \node[left] at (7) {7};
      \node[right] at (0) {0};
      \node[below] at (9) {9};
      \draw[->] (1)-- (-0.25,1);
      \draw[->] (2)-- (1,0);
      \draw[->] (3)-- (2.25,1);
      \draw[->] (4)-- (1.75,2.5);
      \draw[->] (5)-- (0.25,2.5);
      \draw[->] (1)-- (0,1.875);
      \draw[->] (2)-- (0.25,0.375);
      \draw[->] (3)-- (1.75,0.375);
      \draw[->] (4)-- (2,1.875);
      \draw[->] (5)-- (1,2.625);
      \draw[->] (0)-- (0.8,1.65);
      \draw[->] (7)-- (1,1.25);
      \draw[->] (9)-- (1,1.75);
      \draw[->] (6)-- (0.85,1.4);
      \draw[->] (8)-- (1.282,1.4);
    \end{tikzpicture}
    \caption{The Petersen graph determines a reducible canonical 
    curve $C$ of genus $6$ in $\PP^5$.}
    \label{fig:ptersengraph}
  \end{figure}

The article \cite{BE}  pays special attention to planar graphs.
By Steinitz's Theorem, a planar trivalent graph $G$ of genus $g$ is the edge graph of a simple $3$-polytope $P$
with $2g-2$ vertices and $g+1$ facets.
Dual to this is a simplical polytope $P^*$ with $g+1$ vertices. The curve $C_G$ is a linear slice
of the face variety of $P^*$, given algebraically by the Stanley--Reisner ideal of $P^*$. 
This is shown in \cite[Section 6]{BE}, which stresses that $\partial P^*$
is a combinatorial model of a K3 surface. This surface should lie between the canonical curve and
a Mukai Grassmannian. The latter is now a degeneration of those in Section \ref{sec4}.
We illustrate this with an example.

\begin{example}[$g=8$] The {\em $3$-dimensional
  associahedron} $P$ is a combinatorial model for the genus $8$ Mukai Grassmannian
   ${\rm Gr}(2,6) \subset \PP^{14}$. Its prime ideal is
  generated by the $15$ quadrics
    $$ \qquad p_{ik} p_{jl} \, = \, p_{ij} p_{kl} \,+ \, p_{il} p_{jk} \qquad {\rm for} \,\,\,
  1 \leq i <j < k < l \leq 6. $$
  This is a reduced Gr\"obner basis with leading terms on the left \cite[Theorem 5.8]{INLA}.
  These  correspond to the {\em crossing diagonals} in a hexagon.
  The ideal $I$ generated by the $15$ monomials is the \emph{Stanley--Reisner ideal} of the polytope dual
  to the associahedron $P$. This appears in
  the ``second proof'' of \cite[Proposition 3.7.4]{AIT}, and is
  well-known in the theory of cluster algebras.
  Slicing the variety of $I$ with a random hyperplane, we obtain a graph curve of genus $8$ in $\PP^7$.
  Its underlying graph $G$ is the edge graph of $P$, which has $14$ vertices, $21$ edges and $9$ facets.
  \end{example}

We now turn to  self-dual configurations and self-dual matroids
arising from graph curves. Let $G$ be a 
$3$-connected trivalent graph with $2g-2$ vertices.
We saw in Proposition  \ref{prop:graphcurveunique}
that $G$ specifies a projectively unique curve $C_G$ consisting
of $2g-2$ lines in $\PP^{g-1}$.
Every hyperplane $H \simeq \PP^{g-2}$ in $\PP^{g-1}$
determines a canonical divisor  $X_G = H \cap C_G$. 
We assume that $H$ is generic, so that $X_G$
is a generic hyperplane section of $C_G$. Thus, $X_G$ is a configuration of $2g-2$ distinct points $ L_i \cap H$, one for each vertex $i$ of $G$. 
In our computations,
we always take $C_G$ and $H$ to be defined over $\QQ$, so the
$2g-2$ points in $X_G$  have coordinates in $\QQ$.

We write $M_G$ for the rank $g-1$ matroid 
on $2g-2$ elements that is determined by the
canonical divisor $X_G$ in $\PP^{g-2}$. 
The following theorem justifies the notation $M_G$.

\begin{theorem} \label{thm:MG}
  The matroid $M_G$ is self-dual and depends only on the graph $G$.
\end{theorem}

\begin{proof}
Riemann--Roch for  the curve $C_G$  is equivalent to the self-duality of the matroid $M_G$. 
  Since $H$ is generic,
  the matroid is independent of $H$, i.e.~it depends only on the curve $C_G$.
  Proposition~\ref{prop:graphcurveunique} ensures that
  $M_G$ is determined by the graph~$G$.
\end{proof}

\begin{example}[Petersen Graph]\label{ex:petersen}
Fix the graph $G$ in Figure \ref{fig:ptersengraph}.
  We consider the ten $6 \times 3$ submatrices of ${\rm Cyc}_G$, as computed in Example \ref{ex:petersenpart1}, corresponding to the vertices $0,1,\ldots,9$.
  The relevant triples of columns are
  labeled by $[12,51,16], [12,23,27], \ldots, [79,96,49], [80,07,50]$.
  One checks that each of these ten $6 \times 3$ matrices has rank two,
  so its column span is a line in $\PP^5$. Our curve $C_G$ is
  the configuration of these ten lines. We now intersect $C_G$ with
  the hyperplane $H = (6,9,5,1,3,3)^\perp$. We obtain $10$ points in $H \subset \PP^5$.
  We delete the last coordinate. This gives us the canonical divisor as the following configuration of
  ten points:
  \begin{equation}
  \label{eq:petersenmatrix}
   X_G \quad = \quad
  \begin{small} \begin{pmatrix}
      7 & -5 & -1 & -3 & -5 & 0 & 0 & 0 & 0 & 0 \\
      0 & 0 & 0 & 0 & 0 & 14 & -5 & -1 & 1 & 5 \\
      -3 & -3 & 0 & 0 & 0 & -9 & 12 & -1 & 0 & 5 \\
      -3 & 5 & -2 & 0 & 0 & -9 & 0 & 22 & 0 & 0 \\
      -3 & 5 & 1 & 1 & 0 & -9 & -5 & -1 & -3 & 5 \\
    \end{pmatrix}. \end{small}
  \end{equation}
  Among the $\binom{10}{5} = 252$ five-tuples of columns,
  precisely $12$ are nonbases.  The nonbases are  the $12$ five-cycles in the Petersen graph. 
  We encode the matroid by its set of nonbases:
  $$ M_G \,=\, \{
  {12345}, {12368}, {12570}, {12679}, {14569}, {15680}, {23479}, {23780}
  , {34580}, {34689}, {45790}, {67890} \}. $$
  The matroid $M_G$ is self-dual, because
  the complement of every five-cycle in $G$ is a five-cycle.
\end{example}

Theorem \ref{thm:MG} raises the question how 
the matroid $M_G$ can be read off from the graph~$G$.
A purely combinatorial rule is desirable.
In what follows, we present what we know~currently.
The vertices of $G$ encode the $2g-2$ lines of the curve $C_G$.
We write $L_i$ for the line 
indexed by the vertex $i$ in $G$ and
$p_i = L_i\cap H$ for its intersection point with a generic hyperplane $H$.
These points form a self-dual configuration in $\PP^{g-2}$, and $M_G$
is their self-dual matroid.
Any linear dependency among the points $p_i$ arises
from a dependency among the lines $L_i$.
A set of $k$ lines in $C_G$ is dependent if and only if the span of the lines lies inside a subspace $\PP^{k-1}$.

\begin{lemma}\label{lem:circuits}
The size $k$ circuits of the matroid $M_G$ correspond to sets of $k$ lines in the graph curve $C_G$ 
such that the set spans a $\PP^{k-1}$ but every subset of size $\ell<k$ spans at least a $\PP^{\ell}$.
\end{lemma}

\begin{proof}
  This follows directly from the description of $M_G$ as the matroid of the points $p_i$.
\end{proof}

We next discuss the combinatorial identification of some circuits in $M_G$.

\begin{lemma} \label{lemma:cycle-dependent}
  Every $k$-cycle in the graph $G$ gives rise to a dependent set of size $k$ in 
  the matroid $M_G$. In particular, a minimal $k$-cycle of $G$ gives a circuit of $M_G$.
\end{lemma}

\begin{proof}
Let $v_1,\ldots,v_k$ be vertices of $G$ that form a  $k$-cycle.
For the corresponding lines,
   $L_i$ intersects $L_{i+1}$ for $1\leq i < k$, and $L_k$ intersects $L_1$. The span of $L_1,\dots,L_{k-1}$ lies 
   in some $\PP^{k-1}$. Since $L_k$ intersects both $L_1$ and $L_{k-1}$, the line $L_k$ must lie in the span of $L_1$ and $L_{k-1}$. Hence, $L_1,\dots,L_{k}$ all lie in the same $\PP^{k-1}$, so their corresponding points in the hyperplane section form a dependent set in $M_G$.
   
  
  Given a minimal $k$-cycle, any subset of $\ell<k$ vertices represent $\ell$ lines spanning a $\PP^{\ell}$. By Lemma~\ref{lem:circuits}, the corresponding points in the hyperplane section form a circuit of $M_G.$
\end{proof}

\begin{lemma}\label{lem:3cycles}
  Circuits of size $3$ in $M_G$ are in 1-to-1 correspondence with $3$-cycles in~$G$.
\end{lemma}

\begin{proof}
  Lemma \ref{lemma:cycle-dependent} implies that $3$-cycles in $G$ are circuits in $M_G$.
For the converse, suppose that $\{i,j,k\}$ is a circuit in $M_G$.
Then the corresponding lines 
   $L_i, L_j$, and $L_k$ are coplanar. So, they must intersect pairwise, and thus come
    from a $3$-cycle in the graph $G$.
\end{proof}

The next two propositions identify some dependent sets in $M_G$ less apparent in~$G$.

\begin{proposition} \label{prop:CD}
  Let $G$ be a $3$-connected trivalent graph and
  $M_G$ its associated matroid. Let $C$ be a subset of
  $V = \{1,2,\ldots,2g-2\}$  such that, for some
  proper subset $D$ of $V \backslash C$,
  the induced subgraph on $C \cup D$ has 
  genus at least $|D| + 1$. Then $C$ is a dependent set in $M_G$.
\end{proposition}

\begin{proof}
  By Lemma \ref{lemma:cycle-dependent}, given a $k$-cycle in $G$, the span of the corresponding $k$ lines in the graph curve lies inside a $\PP^{k-1}$. 
  Let $|C| = m$ and $|D| = n$. The fact that the subgraph induced by $C \cup D$ has genus $g' \geq n + 1$ means that there are $g'$ linearly independent cycles in $C \cup D$. By induction on $g'$, the span of the $m+n$ lines in $C_G$ that correspond to $C \cup D$ lies inside a $\PP^{m+n - g'} \subseteq \PP^{m-1}$. In particular, the $m$ lines corresponding to $C$ lie in this $\PP^{m-1}$, so 
  $C$ is indeed a dependent set of $M_G$.
\end{proof}

Suppose now that a minimal $\ell$-cycle $C$ of the graph $G$ intersects a minimal $k$-cycle $D$ in $\ell-1$ vertices. 
  Then any subset of $k$ vertices in $C\cup D$ is dependent in $M_G$. In lieu of Lemma \ref{lemma:cycle-dependent}, we prove the following stronger combinatorial statement about matroids.

\begin{proposition}
  \label{prop:containment1}
  If $C$ and $D$ are circuits in a matroid $M$, with $|C| = l$, $|D| = k$, and $|C \cap D| = l-1$, then any size $k$ subset of $C \cup D$ is dependent in $M$.
\end{proposition}

\begin{proof} 
We know that $|C \cup D| = |C| + |D| - |C \cap D| = k+1.$ 
So a size $k$ subset $S \subset C \cup D$ is either equal to $D$, which is dependent, or is of the form $(C \cup D) \setminus \{i\}$, for some $i \in D$. In the latter case, if $i \in D\setminus C$, then $S$ contains the circuit $C$, hence is dependent. Otherwise, we must have $i \in C \cap D$; in this case, $S$ contains a circuit by the third circuit axiom, so $S$ is again dependent. 
\end{proof}

 Propositions \ref{prop:CD} and \ref{prop:containment1} can be used to find intricate dependencies in $M_G$ by looking at the combinatorics of the graph $G$. More generally, for a set of cycles in $G$, we can analyze their intersection behavior and
 draw conclusions about the dependencies among the lines indexed by these cycles.
 From this we can infer dependent sets of $M_G$.  
 The procedure is illustrated in the following example.

\begin{example}\label{ex:ralucagraph}
  Let $G$ be the graph for $g=6$ on the left of Figure \ref{fig:ralucalgraph}. 
  The rank $5$ matroid $M_G$ has $162$ bases out of the $252$ quintuples.
  A computation shows that the set $C = \{2,3,8,9\}$ is a circuit of $M_G$. However,
  $C$ does not satisfy the condition in Proposition \ref{prop:CD}.
  There is no complementary set $D$ such that $C \cup D$ induces
  a subgraph of genus $\geq |D|+1$.
  
  Let $P_{0,1,4}$ and $P_{5,6,7}$ be the planes through the lines $L_0,L_1,L_4$ and $L_5,L_6,L_7$,  guaranteed by Lemma~\ref{lem:3cycles}. The $\PP^3$ formed by $L_0,L_1,L_2,L_3$ intersects the $\PP^4$ formed by
  $L_0,L_3,L_8,L_9,L_4$ in $P_{0,1,4}$ and in $L_3$ not on that plane,
  so the intersection is the entire $\PP^3$. Hence
  this $\PP^3$ is contained in the~$\PP^4$. 
  Symmetry implies the same containment for the other half of the graph. The lines $L_2,L_3,L_8,L_9$ lie in the intersection of the two $\PP^4$s, so they are contained in a space of dimension $\leq 3$.
  \end{example}

\vspace{-0.2cm}
\begin{figure}[h] \begin{subfigure}[b]{.45\textwidth}\centering
    \begin{subfigure}[c]{.45\textwidth}\centering
      \begin{tikzpicture}[scale = 0.6]
        \coordinate (6) at (0,0);
        \coordinate (7) at (0,2);
        \coordinate (2) at (0,4);
        \coordinate (1) at (0,6);
        \coordinate (8) at (1,2);
        \coordinate (3) at (1,4);
        \coordinate (5) at (2,0);
        \coordinate (9) at (2,2);
        \coordinate (4) at (2,4);
        \coordinate (0) at (2,6);
        \foreach \i in {1,2,3,4,5,6,7,8,9,0} {
          \path (\i) node[circle, black, fill, inner sep=2]{};
        }
        \draw[black] (0)--(1) -- (2) -- (7) -- (6) -- (5) -- (9) -- (8) -- (3) -- (2);
        \draw[black] (3) -- (0) -- (4) ;
        \draw[black]  (6) -- (8) ;
        \draw[black,] (5) to[out=-110,in=-150, distance=3.25cm] (7) ;
        \draw[black] (1) to[out=70,in=30,distance=3.25cm]  (4) --(9);
        \node[left] at (1) {1};
        \node[right] at (3) {3};
        \node[left] at (2) {2};
        \node[right] at (4) {4};
        \node[right] at (5) {5};
        \node[below] at (6) {6};
        \node[left] at (8) {8};
        \node[left] at (7) {7};
        \node[above] at (0) {0};
        \node[right] at (9) {9};
      \end{tikzpicture}
    \end{subfigure}\quad
    \begin{subfigure}[c]{.45\textwidth}\centering
      \begin{tikzpicture}[scale = 0.6]
        \coordinate (5) at (0,0);
        \coordinate (8) at (0,2);
        \coordinate (3) at (0,4);
        \coordinate (1) at (0,6);
        \coordinate (6) at (1,-1);
        \coordinate (0) at (1,7);
        \coordinate (7) at (2,0);
        \coordinate (9) at (2,2);
        \coordinate (2) at (2,4);
        \coordinate (4) at (2,6);
        \foreach \i in {1,2,3,4,5,6,7,8,9,0} {
          \path (\i) node[circle, black, fill, inner sep=2]{};
        }
        \draw[black] (4)--(0) -- (1) -- (4) -- (2) -- (3) -- (8) -- (9) -- (7) -- (5) -- (6)--(7);
        \draw[black] (3) -- (1) ;
        \draw[black] (2) -- (9) ;
        \draw[black] (5) -- (8) ;
        \draw[black] (0) to[out=20, in=-20, distance = 4cm] (6) ;
        \node[left] at (1) {1};
        \node[left] at (3) {3};
        \node[right] at (2) {2};
        \node[right] at (4) {4};
        \node[left] at (5) {5};
        \node[below] at (6) {6};
        \node[left] at (8) {8};
        \node[right] at (7) {7};
        \node[above] at (0) {0};
        \node[right] at (9) {9};
      \end{tikzpicture}
    \end{subfigure}
    \caption{Two graphs for a matroid with $162$ bases.}\label{fig:ralucalgraph}
  \end{subfigure}
  \begin{subfigure}[b]{.45\textwidth}
    \begin{subfigure}[c]{.45\textwidth}
      \begin{tikzpicture}[scale = 0.65]
        \coordinate (4) at (0,0);
        \coordinate (3) at (0,2);
        \coordinate (2) at (0,4);
        \coordinate (9) at (-1,2);
        \coordinate (5) at (1,-1);
        \coordinate (1) at (1,5);
        \coordinate (6) at (2,0);
        \coordinate (7) at (2,2);
        \coordinate (0) at (2,4);
        \coordinate (8) at (1,2);
        \foreach \i in {1,2,3,4,5,6,7,8,9,0} {
          \path (\i) node[circle, black, fill, inner sep=2]{};
        }
        \draw[black] (6) -- (4)--(5) -- (6) -- (7) -- (3) -- (2) -- (0) -- (1) -- (2) ;
        \draw[black] (0) -- (7) ;
        \draw[black] (3) -- (4) ;
        \draw[black] (9) to[out=60, in=120, distance = 1cm] (8) ;
        \draw[black] (9) to[out=90, in=160, distance = 2.5cm] (1) ;
        \draw[black] (5) to[out=-160, in=-90, distance = 2.5cm] (9) ;
        \node[above] at (1) {1};
        \node[left] at (3) {3};
        \node[left] at (2) {2};
        \node[left] at (4) {4};
        \node[below] at (5) {5};
        \node[right] at (6) {6};
        \node[below] at (8) {8};
        \node[right] at (7) {7};
        \node[right] at (0) {0};
        \node[left] at (9) {9};
      \end{tikzpicture}
    \end{subfigure}\quad
    \begin{subfigure}[c]{.45\textwidth}
      \begin{tikzpicture}[scale = 0.6]
        \coordinate (1) at (0,0);
        \coordinate (9) at (0,2);
        \coordinate (3) at (0,4);
        \coordinate (4) at (0,6);
        \coordinate (2) at (1,-1);
        \coordinate (6) at (1,7);
        \coordinate (0) at (2,0);
        \coordinate (8) at (2,2);
        \coordinate (7) at (2,4);
        \coordinate (5) at (2,6);
        \foreach \i in {1,2,3,4,5,6,7,8,9,0} {
          \path (\i) node[circle, black, fill, inner sep=2]{};
        }
        \draw[black] (0)--(2) -- (1) -- (0) -- (8) -- (7) -- (5) -- (4) -- (6) -- (5);
        \draw[black] (4) --(3) -- (9)-- (1);
        \draw[black] (3) -- (8) ;
        \draw[black] (7) -- (9) ;
        \draw[black] (6) to[out=20, in=-20, distance = 4cm] (2) ;
        \node[left] at (1) {1};
        \node[left] at (3) {3};
        \node[below] at (2) {2};
        \node[left] at (4) {4};
        \node[right] at (5) {5};
        \node[above] at (6) {6};
        \node[right] at (8) {8};
        \node[right] at (7) {7};
        \node[right] at (0) {0};
        \node[left] at (9) {9};
      \end{tikzpicture}
    \end{subfigure}
    \caption{Two graphs for a matroid with $180$ bases.}
  \end{subfigure}
  \caption{Two non-isomorphic graphs 
  $G$ and $G'$ can give the same matroid $M_G = M_{G'}$.}\label{fig:diffgraphs}
\end{figure}

The two propositions  allow us to infer many dependencies in the
matroid $M_G$. But, they do not yet give an algorithm for
deriving $M_G$ from $G$.
Example \ref{ex:ralucagraph} shows that this derivation is not easy.
The best method we know is to  slice the
graph curve $C_G$ with  a generic hyperplane. It remains an open problem
  to give a combinatorial rule for determining 
  the matroid $M_G$  from the graph $G$.
  Such a rule should show directly that $M_G$ is self-dual.

The main result of this section is a comprehensive
study of all graph curves with $g \leq 7$. The following theorem
summarizes our findings. The data with many details, including
realizations of the matroids $M_G$ for all relevant graphs $G$, can be found at \cite{MathRepo}.

\begin{theorem} \label{thm:graphcurves4567}
  For genus $g=4,5,6,7$, there are respectively $\,2,4,14,57$ distinct graph curves $C_G$.
  The underlying $3$-connected trivalent graphs $G$ yield $2,4,12,45$ distinct self-dual matroids $M_G$ of
  ranks $3,4,5,6$.
  A more detailed summary is given in the discussion below.
\end{theorem}

\begin{proof}[Proof and discussion]
For genus $g$, the matroid has rank $g-1$.
We consider the four~cases.
The graphs are  from
\url{https://en.wikipedia.org/wiki/Table_of_simple_cubic_graphs}.

$\mathbf{g=4}.$ There are two trivalent 3-connected graphs on 6 vertices. They realize the two self-dual matroids of rank three:
the uniform matroid and the one with two disjoint nonbases.

\smallskip
$\mathbf{g=5.}$ \label{ex:graphs4}
We have four graphs $G$ with eight vertices, three planar and one non-planar. 
They are shown in  Figure \ref{fig:graphs8}.
Their rank $4$ matroids $M_G$ have
$54, 60, 64, 66 $ bases.  These matroids are called 
\texttt{4.16.a}, \texttt{4.10.b},  \texttt{4.6.a} and \texttt{4.4.a}
in  Table \ref{tab:rank4} and Figure \ref{fig:graphs8}.
These data appeared in Section~\ref{sec3}, where we
realized self-dual matroids  by
canonical curves. See Remark \ref{rem:dreifuenf}.

\smallskip
$\mathbf{g=6.}$\label{ex:graphsgenus6}
There are $14$ trivalent, 3-connected graphs $G$ of genus $6$. They induce $12$ distinct matroids $M_G$.
These  have the following numbers of bases,  out of $252 = \binom{10}{5}$
possible 5-tuples:
\begin{center}
  \begin{tabular}{c|cccccccccc}
    no. bases & $ 162 $& $180$&$ 192$&$ 198$&$ 200$ & $210$ & $212$ &$222$ & $228$& $240$\\ \hline 
    no. matroids & 2 & 2 & 1 & 1 & 1 & 1 & 1 & 1 & 1 & 1
  \end{tabular}
\end{center}
Here we observe for the first time that two non-isomorphic graphs can 
give the same matroid.
The two pairs of genus six graphs 
with this property are depicted in Figure \ref{fig:diffgraphs}.

\smallskip
$\mathbf{g=7.}$ 
The $57$ trivalent, 3-connected graphs yield
 $45$ distinct self-dual matroids.
Table \ref{tab:genus7} shows their numbers of bases, out of $924 = \binom{12}{6}$, and the distribution 
of the matroids.  
\end{proof}

\begin{table}[h]
  \begin{tabular}{c|cccccccccc}
    no. bases & $  486$ & $ 540$ &$576$& $594 $ & $ 600$ &  $630$ & $ 636$ & $ 640$ & $660$ &  $664$  \\\hline 
    no. matroids & 7 & 8& 4 & 6 & 3 &1 & 1 & 1 & 1 & 1\\
    no. graphs & 5+\textbf{4}+\textbf{2} & 5+\textbf{3}+\textbf{2}+\textbf{2}& 4 & 4+\textbf{2}+\textbf{2} & 2+\textbf{2} &1 & 1 & 1 & 1 & 1\\
    \multicolumn{11}{c}{} \\
    no. bases &   $666$ & $684$ & $714$ & $720$ & $726$ &  $736$ & $744$ &  $750$ &  $756$ &  $774$  \\\hline 
    no. matroids & 4 & 3 & 2&  1 & 1 & 1 & 1 & 1 & 2 & 1 \\
    no. graphs &  4 & 3 & 2&  1 & 1 & 1 & 1 & 1 & 2 & 1 \\
    \multicolumn{11}{c}{} \\
    no. bases &  $784$ &$786$ & $804$ &   $816$ &  $834$ & $ 840$ & & & & \\ \hline 
    no. matroids  & 1 & 1 & 1 & 1 & 1 &  2& & & & \\
    no. graphs & 1 & 1 & 1 &  1 & 1 &  1+\textbf{2}& & & & 
  \end{tabular}
  \caption{The data on graph curves of genus $7$ and their arising matroids. The tuples of graphs that give rise to the same matroid are printed in bold in the last row of the table.}\label{tab:genus7}
\end{table}

  \section{Self-duality in tropical geometry}
  \label{sec6}
  
  In his seminal work on Chow quotients of Grassmannians,
  Kapranov discusses self-dual configurations  \cite[Section 2.3]{Kap},
  and shows that Gale duality induces a morphism
  of Chow quotients \cite[Corollary 2.3.14]{Kap}.
  Following Keel and Tevelev \cite[Section 1.14]{KT},
  this Chow quotient is the tropical compactification of
  the configuration space $\mathcal{R}(n,2n)$.
  This compactification is defined as the closure of
  $\mathcal{R}(n,2n)$ in the toric variety given by
  the tropical Grassmannian \cite[Section 4.3]{MS},
  by way of the construction described in \cite[Section~6.4]{MS}.
  
  The restriction of the Chow quotient to the self-dual Grassmannian
  thus yields a natural compactification of 
  the self-dual configuration space $\mathcal{S}(n)$.
  At a combinatorial level, this compactification is described
  by the self-dual locus inside the tropical Grassmannian.
  It is the purpose of this section to study that locus.
  This requires us to lift self-duality from matroids to valuated matroids,
  and hence to tropical linear spaces and matroid subdivisions.

  We now work over the field  of  Puiseux series  $K=\CC\{\!\{t\}\!\}$,
  with the $t$-adic valuation and
    residue field $\CC$.
  Recall from \cite[Section 4.3]{MS} that
  the {\em tropical Grassmannian}, denoted ${\rm trop}({\rm Gr}(n,2n))$, is the tropicalization of the open part ${\rm Gr}(n,2n)^o$. This
  is a balanced polyhedral fan of  dimension $n^2$ in $\RR^{\binom{2n}{n}}\!/\RR {\bf 1}$. Its lineality space  
  has dimension $2n-1$ mod $\mathbb{R}\mathbf{1}$,
   $$ L \,\, = \,\,
   \biggl\{ \bigl(\,\sum_{i\in I}\mu_i \,\bigr)_{\substack{ I \in \binom{[2n]}{n}} }\,\, :
    \,\, \mu_1,\ldots,\mu_{2n} \in \RR
   \biggr\}. $$
   
  The {\em tropical configuration space} for $2n$ points in $(n-1)$-space is the quotient 
$$ {\rm trop}(\mathcal{R}(n,2n)) \,\,:=\,\, {\rm trop}({\rm Gr}(n,2n))/L 
\,= \,{\rm trop}\bigl({\rm Gr}(n,2n)/(K^*)^{2n} \bigr).$$ 
This is a polyhedral fan of dimension $(n-1)^2$ which is pointed, i.e.~its lineality space is $\{0\}$.

  The tropical Grassmannian is naturally contained in the tropical prevariety $ {\rm Dr}(n,2n)$,
  known as the {\em Dressian}. Points of the Dressian are {\em valuated matroids}.
  These are the tropical solutions of the quadratic P\"ucker relations.
  The containment $ {\rm trop}({\rm Gr}(n,2n)) \subseteq {\rm Dr}(n,2n) $ is strict when $n \geq 4$.  Each 
  valuated matroid $q \in{\rm Dr}(n,2n)$ defines a tropical linear space
  that is dual to a matroid subdivision of the {\em hypersimplex}
  $\Delta(n,2n)$, the matroid polytope of the uniform matroid of rank $n$ on $2n$ elements.
  These concepts are explained in \cite[Section 4.4]{MS}.
  
  \begin{example}
    For $n=2$, the fan ${\rm trop}(\mathcal{R}(2,4))$ is $1$-dimensional, and has
    three rays, one for each of the three labelings on a general tropical line in $3$-space. This space coincides with the tropical moduli space $\mathcal{M}^{\text{trop}}_{0,4}$ of 4-marked curves of genus $0$.
    The Dressian $ {\rm Dr}(2,4)$ coincides with the tropical Grassmannian. It parametrizes tropical lines in 3-space. These lines are dual to the three matroid subidivisions of the octahedron $\Delta(2,4)$.  Each subdivides the octahedron into two square-based pyramids.
    This is illustrated in \cite[Figure 4.4.1]{MS}.
  \end{example}
  
  We next define the self-dual loci within these tropical spaces.
  The tropical Pl\"ucker coordinates are denoted $q_I$,
  where $I$ runs over the set $\binom{[2n]}{n}$ of $n$-element subsets of $[2n]$.
  We write $q^*_I = q_{I^c}\,$ for its dual Pl\"ucker coordinate. 
  Unlike the Hodge star operation in (\ref{eq:hodgestar})
  for classical Pl\"ucker coordinates $p_I$,
  no signs are needed for the tropical Hodge star map.
  In terms of valuated matroids, the tropical Hodge star sends a valuated matroid to its dual. 
 
    A vector $q$ in $\RR^{\binom{2n}{n}}$  is  {\em self-dual}
    if there exist  $\mu_1,\mu_2,\ldots,\mu_{2n} \in \RR$ such that 
    \begin{equation}
      \label{eq:selfdual}
      \,q_{i_1 i_2 \cdots i_n}  + 
      \mu_{i_1} + \mu_{i_2} + \cdots + \mu_{i_n} \,\,= \,\,\,q^*_{i_1 i_2 \cdots i_n} \,\,\quad \hbox{for all} \quad 1 \leq i_1 < i_2 < \cdots < i_n \leq 2n. 
    \end{equation}
    Thus, $q$ is self-dual if and only if the 
    difference  vector $q-q^*$ lies in the lineality space $L$.
\medskip

  Eliminating the unknowns $\mu_j$ in
  \eqref{eq:selfdual} leads to a linear subspace $L_{\rm sd}(n,2n)$ 
  that has codimension $2n-1$ in $\RR^{\binom{2n}{n}}$. It is spanned by $L$ and the
  $\frac{1}{2}\binom{2n}{n}$ vectors $e_I + e_{I^c}$.
Its equations are the binomials in
(\ref{eq:sixpointsonconic})  and Proposition \ref{prop:toricideal}, but written additively.
To compute them, fix
a matrix whose kernel equals $L$,
 and multiply it with the $\binom{2n}{n}\times \binom{2n}{n}$ matrix with $+1$ on its diagonal and $-1$ on its anti-diagonal. The kernel of the resulting matrix equals $L_{\rm sd}(n,2n)$.

  Since ${\rm trop}(\mathcal{R}(n,2n))$ is the quotient of 
  $ {\rm trop}({\rm Gr}(n,2n))$ modulo its lineality space $L$,
  self-duality is well-defined within the tropical configuration space.
  The Hodge star $q \mapsto q^*$~defines an involution on
  ${\rm trop}(\mathcal{R}(n,2n))$. Its fixed points are gained by intersecting with $L_{\rm sd}(n,2n)/L$.

We write ${\rm trop}({\rm Gr}(n,2n))^{\rm sd}$,
${\rm Dr}(n,2n)^{\rm sd}$ and  ${\rm trop}(\mathcal{R}(n,2n))^{\rm sd}$ for the intersections of
${\rm trop}({\rm Gr}(n,2n))$, ${\rm Dr}(n,2n)$ and  ${\rm trop}(\mathcal{R}(n,2n))$
 with  $L_{\rm sd}(n,2n)$ and $L_{\rm sd}(n,2n)/L$, respectively.
We refer to the points in these intersections as {\em self-dual valuated matroids}. 
They are precisely the fixed points under the tropical Hodge star involution
$q \mapsto q^*$ 
modulo~$L$.
\medskip
  
Our next result establishes the connection to matroid subdivisions of the hypersimplex.
  
  \begin{proposition} \label{prop:involution}
    Fix a self-dual valuated matroid $q \in {\rm Dr}(n,2n)^{\rm sd}$ 
    and let $\mathcal{M}(q)$ be the set of rank $n$ matroids
    that appear  in the associated matroid subdivision of the hypersimplex $\Delta(n,2n)$. 
    The set $\mathcal{M}(q)$ is fixed under taking matroid duals, i.e. $\mathcal{M}(q)=\{M^* \,|\, M \in \mathcal{M}(q)\}$. 
  \end{proposition}
  
  \begin{proof}
    Let $M$ be a matroid in $\mathcal{M}(q)$. Recall that we write $M^*$ for the dual matroid of $M$ obtained by applying set complementation $I \mapsto I^c$ to the bases. We show that
    $M^* \in \mathcal{M}(q)$. 
    Following \cite[Lemma 4.4.6]{MS}, the matroid polytope $P_M$ is
    a cell in the regular 
    subdivision of  the hypersimplex $\Delta(n,2n)$ defined by $q$.
    This means that, after adding a vector in the lineality space to $q$, we can assume that
    $q$ is non-negative and $q_I = 0$ if and only if $I$ is a basis of $M$.
    The Hodge star $q^*$  is also non-negative, and $q^*_J = 0$ if and only if $J $ is a basis of $M^*$, 
    i.e.~the matroid polytope of $M^*$ is a cell in the subdivision defined by~$q^*$.
    By (\ref{eq:selfdual}), $q$ and $q^*$ differ by a vector in the lineality space, 
    so they define the same 
    subdivision of $\Delta(n,2n)$.
  \end{proof}
  
  \begin{example}[$n=2$]
    Fix the self-dual valuated matroid $q = (q_{12} ,q_{13},q_{14},q_{23},q_{24},q_{34}) = (1,0,0,0,0,1)$.
    Then $\mathcal{M}(q)$ has two maximal cells.
        In terms of bases, these rank $2$ matroids are
        $M_1=\{12,13,14,23,24\}$ and $M_2=\{13,14,23,24,34\}$. Neither of them is self-dual, but
    the involution $I \mapsto I^c$ swaps $M_1$ and $M_2$. 
    The intersection $M_1 \cap M_2$ is a self-dual matroid.
  \end{example}
  
  The {\em tropical self-dual Grassmannian} ${\rm trop}({\rm SGr}(n,2n))$ is the tropicalization of the open part of the self-dual Grassmannian.   
  Combining Theorem \ref{thm:SOn} with \cite[Theorem 3.3.5]{MS} yields:
    
  \begin{corollary}
    The tropical self-dual Grassmannian $ {\rm trop}({\rm SGr}(n,2n)) $ is
    a balanced fan of pure dimension $\binom{n}{2} +2n-1$ in $\,\RR^{\binom{2n}{n}}/\RR {\bf 1}$.
    Its lineality space has  dimension $2n-1$. The~quotient modulo that  space equals
    $ \,{\rm trop}(\mathcal{S}(n))$. This is a pointed fan of pure dimension~$\binom{n}{2}$.
  \end{corollary}
  
  The tropical prevarieties arising from the intersection with $L_{\rm sd}(n,2n)$
 are outer approximations to the tropicalizations of
  ${\rm trop}({\rm SGr}(n,2n))$ and $ {\rm trop}(\mathcal{S}(n))$. 
  We have the inclusions ${\rm trop}({\rm Gr}(n,2n))^{\rm sd} \supseteq {\rm trop}({\rm SGr}(n,2n))$,
  and therefore also  $\trop(\mathcal{R}(n,2n)^{\rm sd}) \supseteq {\rm trop}(\mathcal{S}(n))$.
  However, for $n \geq 4$, we do not know whether these inclusions are equalities.

We now turn to the first non-trivial scenario, namely $n=3$. Here, we show that the fixed point locus ${\rm trop}({\rm Gr}(3,6))^{\rm sd}$ coincides with ${\rm trop}({\rm SGr}(3,6))$. We give a 
parametrization of this tropical variety by  the space $\trop( \Gr(2,6))$
of phylogenetic trees with six leaves.

Recall that the Grassmannian ${\rm Gr}(2,6)$ is $8$-dimensional and lives in $\PP^{14}$, with Pl\"ucker coordinates $(p_{12}:p_{13}:\ldots:p_{56})$. In \cite{PS} Pachter and Speyer studied the monomial map 
\begin{equation}
\label{eq:PachterSpeyer}
    \varphi: {\rm Gr}(2,6) \dashrightarrow {\rm Gr}(3,6) ,\,\, p \mapsto (p_{ij} p_{ik} p_{jk})_{ijk}.
\end{equation}
  This map parametrizes 6-tuples of points on a conic in $\PP^2$. It maps birationally onto ${\rm SGr}(3,6)$.

  \begin{theorem}  \label{thm:dreisechs}
    The following is a balanced fan of pure dimension $8$ in $\RR^{20}/\RR {\bf 1}$:
    \begin{equation}
      \label{eq:dreisechs1}  {\rm trop}({\rm SGr}(3,6))\, = \,
     {\rm trop}({\rm Gr}(3,6))^{\rm sd}  \,= \,
     {\rm Dr}(3,6)^{\rm sd}  . 
    \end{equation}
    This tropical variety is linearly isomorphic to the  tropical Grassmannian $\,{\rm trop}({\rm Gr}(2,6))$.
    The isomorphism is given by  the tropical Pachter--Speyer map
     $\,q \mapsto (q_{ij}+ q_{ik} +q_{jk})_{ijk}\,$ which maps this tree space onto  (\ref{eq:dreisechs1}). 
Modulo lineality spaces we obtain the $3$-dimensional pointed~fan
\begin{equation}
      \label{eq:dreisechs2} {\rm trop}(\mathcal{R}(2,6)) = {\rm trop}(\mathcal{S}(3)) \,= \, {\rm trop}(\mathcal{R}(3,6))^{\rm sd}.
    \end{equation}
This is the fan over the 2-dimensional simplicial complex with $f$-vector $(25,105,105)$.
  \end{theorem}
  Here we consider all the fans with their unique coarsest fan structure.
  \begin{proof}
    The right equality of (\ref{eq:dreisechs1}) follows from
    $ {\rm Dr}(3,6)\, = \,  {\rm trop}({\rm Gr}(3,6))$ \cite[Example~4.4.10]{MS}.
    The left equality  of (\ref{eq:dreisechs1}) is verified using computational methods. 
    Recall that ${\rm SGr}(3,6)$ is the $8$-dimensional subvariety of $\PP^{19}$ defined by the quartic binomial \eqref{eq:sixpointsonconic} plus the $35$ quadratic Pl\"ucker relations for ${\rm Gr}(3,6)$.
    Using these $36$ equations and the {\tt Singular} package for tropical varieties, we compute
   $\trop(\SGr(3,6))$. We find that this is an $8$-dimensional fan, with $5$-dimensional lineality space, and  
   the quotient is  simplicial with f-vector $(25,105,105)$. 
    Using \texttt{polymake}, we intersect the Dressian $\Dr(3,6)$
    with $L_{\rm sd}(3,6)$ to get $\Dr(3,6)^{\rm sd}$.
    Comparing lineality spaces, rays, and cones up to $S_6$-symmetry, we conclude 
     $ {\rm trop}({\rm SGr}(3,6)) \,=\,
    {\rm Dr}(3,6)^{\rm sd}$.
        
    The monomial map (\ref{eq:PachterSpeyer})  becomes a linear
    isomorphism after tropicalization  by \cite[Corollary 3.2.13]{MS}.
    The identities in (\ref{eq:dreisechs2}) follow from the left one in (\ref{eq:dreisechs1}) and
    the tropical Pachter--Speyer map. Finally,  ${\rm trop}({\rm Gr}(2,6)) $ has 
    f-vector $(25,105,105)$ by \cite[Example 4.3.15]{MS}.
  \end{proof}

  We next analyze the identification of the  self-dual tropical Grassmannian $\trop(\SGr(3,6))$ 
  with the tree space $\trop(\Gr(2,6))$ in more detail, starting with
  $\trop(\text{Gr}(3,6))^{\rm sd}$.
    Recall from \cite[Example 4.4.10]{MS}
  that $\trop(\Gr(3,6))$ has three types of rays:
    \begin{align*}
      &\text{type E:} \,\, e_{i_1i_2i_3}, \qquad  \qquad
      \text{type F:} \,\, f_{i_1i_2i_3i_4} = e_{i_1i_2i_3}+e_{i_1i_2i_4}+e_{i_1i_3i_4}+e_{i_2i_3i_4},\\
      &\text{type G:} \,\, g_{i_1i_2i_3i_4i_5i_6} = f_{i_1i_2i_3i_4} + e_{i_3i_4i_5}  + e_{i_3i_4i_6}, \,\,\, 
      \qquad \text{ where } i_1,\ldots,i_6 \in [6].
    \end{align*}
The intersection $\trop(\text{Gr}(3,6))^{\rm sd} = \trop(\Gr(3,6))  \cap L_{\rm sd}(3,6) $ has
  two types of rays: the $15$ rays of type F and the $10$ rays
 $e_{I}+e_{I^c}$, where $I \in \binom{[6]}{3}$. We denote the latter type by   E$^{\rm sd}$. 

The maximal cones of  $\trop(\Gr(3,6))$ come in seven types, 
shown in \cite[Figure 5.4.1]{MS}. Fix such a cone $C$.
The intersection $C \cap L_{\rm sd}(3,6)$
is spanned by the rays of type F and E$^{\rm sd}$ that lie in $C$. 
Only the cones FFFGG and EEFF(a) have maximal-dimensional intersection with $L_{\rm sd}(3,6)$.
The cones EEFF(b), EFFG and EEFG intersect in proper faces, which are intersections with cones of the form FFFGG and EEFF(a). 
For the  cones  EEEE and EEGG, the intersection equals the lineality space of $\trop(\Gr(3,6))$.
Thus, $\trop(\text{Gr}(3,6))^{\rm sd}$ consists of the $105 = 15+90$ cones
$C \cap L_{\rm sd}(3,6)$ where $C$ has type FFFGG or EEFF(a).
    
    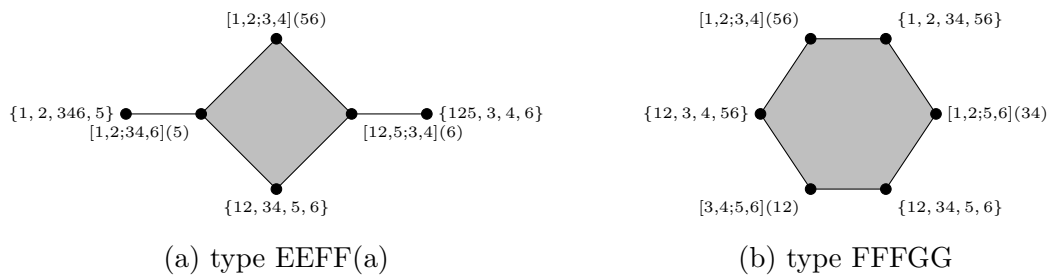
\begin{figure}[h]
      \centering
      \begin{subfigure}{0.45\textwidth}
        \centering
        \begin{tikzpicture}
          \draw[black, fill = lightgray] (0,0) node[below] {\tiny$\{12,34,5,6\}$}
          --(1,1) node[below right] {\tiny [12,5;3,4](6)}
          -- (0,2) node[above] {\tiny[1,2;3,4](56)}
          -- (-1,1) node[below left] {\tiny[1,2;34,6](5)}-- cycle;
          \draw[black] (-2,1) node[left] {\tiny$\{1,2,346,5\}$}-- (-1,1) ;
          \draw[black] (1,1) -- (2,1) node[right]{\tiny$\{125,3,4,6\}$} ;
          \filldraw [black] (0,0) circle (2pt);
          \filldraw [black] (1,1) circle (2pt);
          \filldraw [black] (0,2) circle (2pt);
          \filldraw [black] (-1,1) circle (2pt);
          \filldraw [black] (-2,1) circle (2pt);
          \filldraw [black] (2,1) circle (2pt);
        \end{tikzpicture}
        \caption{type EEFF(a)}
      \end{subfigure}
      \begin{subfigure}{0.45\textwidth}
        \centering
        \begin{tikzpicture}
          \node (A) at (0,0) {};
          \node (B) at (1,0) {};
          \node (C) at (5/3,1) {};
          \node (D) at (1,2) {};
          \node (E) at (0,2) {};
          \node (F) at (-2/3,1) {};
          \draw[black, fill = lightgray] (0,0) node[below left]  {\tiny[3,4;5,6](12)}
          --(1,0) node[below right]  {\tiny$\{12,34,5,6\}$}
          -- (5/3,1) node[right] {\tiny[1,2;5,6](34)}
          -- (1,2) node[above right]  {\tiny$\{1,2,34,56\}$}
          -- (0,2) node[above left] {\tiny[1,2;3,4](56)}
          -- (-2/3,1) node[left] {\tiny$\{12,3,4,56\}$}
          --cycle;
          \foreach \i in {A,B,C,D,E,F} { 
            \filldraw [black] (\i) circle (2pt); }
        \end{tikzpicture}
        \caption{type FFFGG}
      \end{subfigure}
      \caption{The two types of the maximal cones in $\text{trop}({\rm SGr}(3,6))$  }
    \label{fig:tropSGr(3,6)}
  \end{figure}

  \begin{figure}[h]
    \centering
    \begin{subfigure}{0.45\textwidth}\centering
      \begin{tikzpicture}[scale = 0.6]
        \draw[black] (0,0) -- (3,0) ;
        \draw[black] (1,0) -- (1,1) ;
        \draw[black] (2,0) -- (2,1) ;
        \draw[black] (0,0) -- (-1,1) ;
        \draw[black] (0,0) -- (-1,-1) ;
        \draw[black] (3,0) -- (4,1) ;
        \draw[black] (3,0) -- (4,-1) ;
        \node (A) at (-1.15,1) {1};
        \node (B) at (-1.15,-1) {2};
        \node (C) at (1,1.2) {3};
        \node (D) at (2,1.2) {4};
        \node (E) at (4.15,1) {5};
        \node (F) at (4.15,-1) {6};
      \end{tikzpicture}
      \caption{caterpillar tree}
       \label{fig:caterpillar}
    \end{subfigure}
    \begin{subfigure}{0.45\textwidth}
      \centering
      \begin{tikzpicture}[scale = 0.6]
        \node (A) at (-1.1,2.1) {1};
        \node (B) at (1.1,2.1) {2};
        \node (C) at (2.2,-1) {3};
        \node (D) at (1,-2.2) {4};
        \node (E) at (-1,-2.2) {5};
        \node (F) at (-2.2,-1) {6};
        \draw[black] (0,0) -- (0,1) ;
        \draw[black] (0,0) -- (1,-1) ;
        \draw[black] (0,0) -- (-1,-1) ;
        \draw[black] (0,1) -- (-1,2) ;
        \draw[black] (0,1) -- (1,2) ;
        \draw[black] (1,-1) -- (2,-1) ;
        \draw[black] (1,-1) -- (1,-2) ;
        \draw[black] (-1,-1) -- (-2,-1) ;
        \draw[black] (-1,-1) -- (-1,-2) ;
      \end{tikzpicture}
      \caption{snowflake tree}
       \label{fig:snowflake}
    \end{subfigure}
    \caption{The two types the maximal cones in the tree space
    $\text{trop}(\text{Gr}(2,6))$.}\label{fig:trees}
  \end{figure}
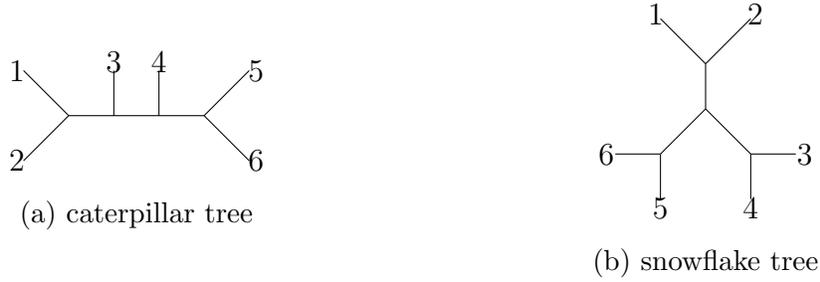

The $25 =15+10$ rays of  ${\rm trop}({\rm Gr}(2,6))$  are indexed by {\em splits}
(cf.~\cite[Example 4.3.15]{MS}) of the set $[6]$.
 Under the tropical Pachter--Speyer map, the $15$ rays for 2-4 splits are 
mapped to F rays, while the $10$ rays for 3-3 splits are mapped
to   E$^{\rm sd}$ rays. This induces the identification of maximal cones.
The $90$ cones of  $\text{trop}(\text{Gr}(2,6))$  indexed by 
{\em caterpillar trees} are mapped to the cones of type  EEFF(a),
while the $15$ cones indexed by {\em snowflake trees} are mapped
to the cones  of type FFFGG. This identification matches the
pictures between Figures \ref{fig:tropSGr(3,6)} and~\ref{fig:trees}.
  
We now turn to degenerations of self-dual point configurations. Our aim is to show
how these degenerations can be understood using the combinatorial machinery developed above.

\begin{example}
  The following self-dual point configuration in  $\PP^2$ appeared in  \cite[Example~1]{AFMS}.
    \begin{align}\label{eq:pts} 
    X \,\, = \,\, \begin{small}
    \begin{pmatrix}
      1 & 1 & 1 & 1 & 1 & 1 \\
      1 & 1+t & 2 & 2+t & 3 & 3+t\\
      1 & (1+t)^2& 4 &(2+t)^2 & 9 & (3+t)^2
    \end{pmatrix}. \end{small}
  \end{align}
This is the section $\{w=0\}$ of the smooth canonical curve in $\PP^3$ defined by
the quadric $\,xz - y^2-w(x-y-z) \,$ and the cubic 
$ (6t^3 {+} 36t^2 {+} 66t {+} 36)x^3 - (11t^3 {+} 84t^2 {+} 193t {+} 132)x^2y $ $
    + (6t^3 {+} 69t^2 {+} 216t {+} 193)x^2z  $ $
- (t^3 {+} 24t^2 {+} 116t {+} 144)xyz + (3t^2 {+} 30t {+} 58)xz^2 - (3t {+} 12)yz^2 + z^3+w^2(y-x)$.
    This point configuration $X$ determines the following point in $\trop(\SGr(3,6))$:
  \begin{equation}
  q\,\,=\,\,  [1,1,1,1,1,0,0,0,0,1,1,0,0,0,0,1,1,1,1,1] \,\, \in \,\,\RR^{20} / \RR {\bf 1}. \label{eq:ptsgrass}
  \end{equation}
  It lies in the cone of type FFFGG and it is the image of the
  following point in   $\trop(\Gr(2,6))$:
  $$[1,0,0,0,0,0,0,0,0,1,0,0,0,0,1] \, = e_{12} + e_{34} + e_{56} \,\,\in \,\,\RR^{15}/\RR {\bf 1}.$$
  This corresponds to the snowflake tree
  labeled  as in Figure \ref{fig:snowflake}.

To understand the compactification of $\mathcal{S}(3)$,
we explore different matrices for
the point given by $X$.
These are obtained from the actions by
${\rm GL}(3,K)$ and $(K^*)^6$.
One example is
\begin{align}\label{eq:pts2} 
X' \,\, = \,\, \begin{small}
    \begin{pmatrix}
      t^{-\frac{    2}{3}} & t^{-\frac{   2}{3}} & t^{\frac{1}{3}} & t^{\frac{1}{3}} & t^{\frac{1}{3}} & t^{\frac{1}{3}}\\
      t^{-\frac{    2}{3}} & (1+t)t^{-\frac{    2}{3}} &2t^{\frac{1}{3}} & (2+t)t^{\frac{1}{3}} & 3t^{\frac{1}{3}} & (3+t)t^{\frac{1}{3}} \\
      t^{-\frac{    2}{3}} & (1+t)^2  t^{-\frac{    2}{3}}&4t^{\frac{1}{3}} & (2+t)^2t^{\frac{1}{3}} & 9t^{\frac{1}{3}} & (3+t)^2t^{\frac{1}{3}} \\
    \end{pmatrix}.
    \end{small}
\end{align}
The tropical Pl\"ucker vector  of this matrix is equivalent to (\ref{eq:ptsgrass}) modulo the lineality space:
\begin{equation} q' \,\,=\,\, [0,0,0,0,1, 0,  0,  0,0,1, 1,0, 0,0, 0,1, 2, 2, 2,2] \,\, \in \,\,
\trop(\SGr(3,6)) \,\subset \,\RR^{20} / \RR {\bf 1}. \label{eq:ptsgrass2}
\end{equation}

Consider what happens as  $t\rightarrow 0$.
In (\ref{eq:pts}) the six points collide pairwise.
This corresponds to the grey triangle picture in the middle of Figure \ref{fig:boundaries} (right).
However, in (\ref{eq:pts2}) only the pairs $\{3,4\}$ and $\{5,6\}$ collide, while
 points $1$ and $2$ stay distinct in the limit.
This is the encircled picture on the upper right of the right image in Figure \ref{fig:boundaries}.
The two pictures show two distinct matroids of rank $3$ on $[6]$, their bases
are given by the zero coordinates in (\ref{eq:ptsgrass}) and (\ref{eq:ptsgrass2}).
\end{example}

  \begin{figure}[h] $\!\!\!\!\!\!\!\!\!\!\!\!$
    \includegraphics[height=0.23\textheight]{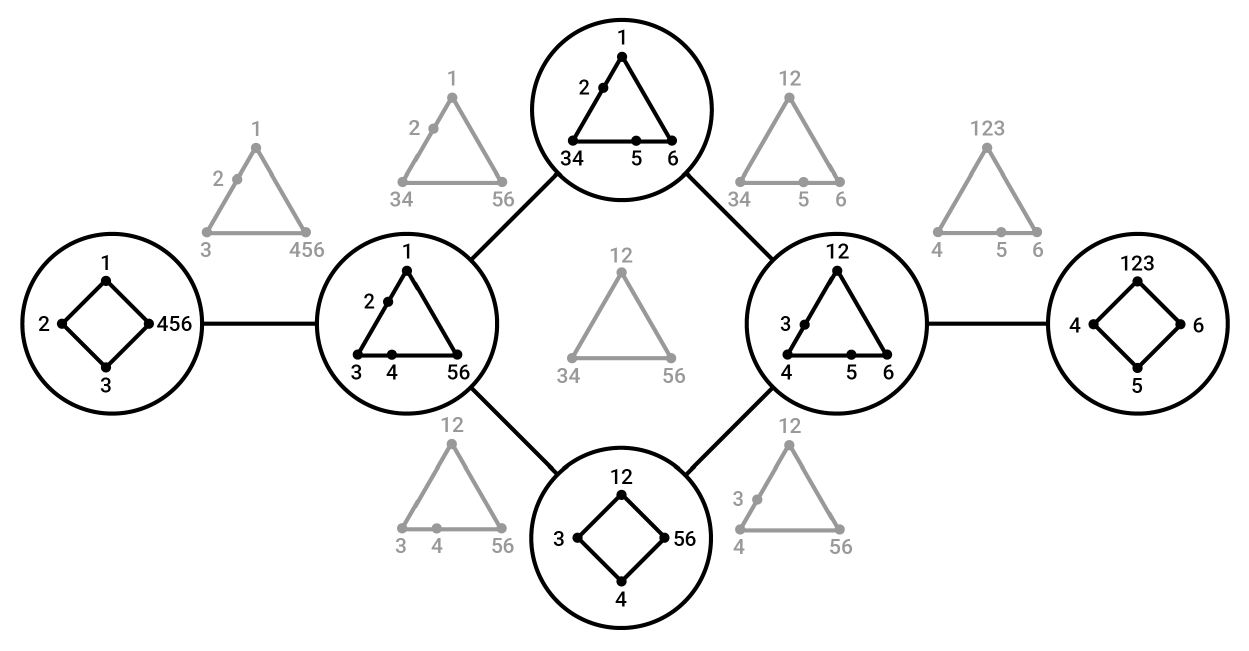}
    \includegraphics[height=0.34\textheight]{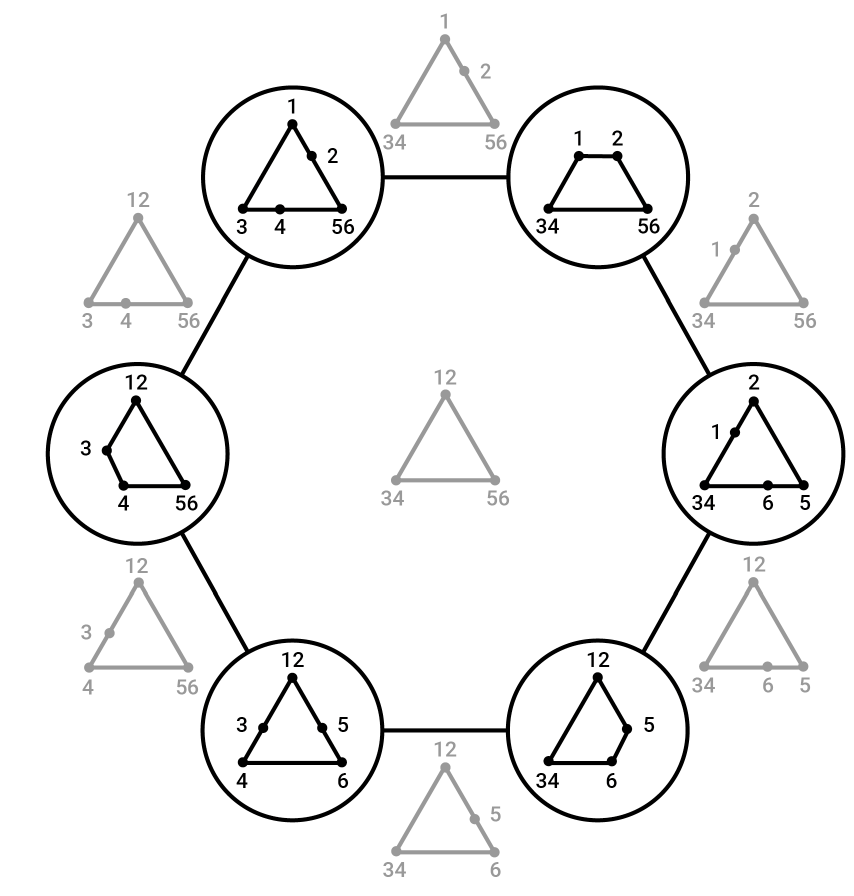}
    \caption{Configurations of type EEFF(a) and type
     FFFGG in the compactification of~$\mathcal{S}(3)$}
     \label{fig:boundaries}
  \end{figure}

Figure \ref{fig:boundaries} is a representation of the self-dual matroid subdivisions
corresponding to the $105$ triangles in Theorem  \ref{thm:dreisechs},
$90$ of type EEFF(a) and $15$ of type FFFGG.
The pictures coincide with those in Figure \ref{fig:tropSGr(3,6)}, but now
we draw configurations for the matroids in 
each subdivision.
The valuated matroid $q$ is self-dual,
and $\mathcal{M}(q)$ satisfies the conclusion of Proposition \ref{prop:involution}.

Our discussion  is meant to explain the compactification of 
the self-dual configuration space $\mathcal{S}(3)$ given by
Kapranov's Chow quotient.
Figure \ref{fig:boundaries} shows the limit behavior for 
configurations of six points (like $X \simeq X'$)
that lie on a conic in $\PP^2$.
It is instructive to compare our pictures to 
those drawn by Schaffler and Tevelev in \cite{SchaffTev}.
Their diagrams represent the limit behavior of configurations in their Mustafin variety, which is a degeneration of $\PP^2$.
For instance, the picture \#20 in \cite[Table 15]{SchaffTev}
shows our configuration of type FFFGG.

We conclude with a discussion of
the case $n=4$, which is considerably more challenging.
Recall from Theorem \ref{thm:gamma}
 that $\mathcal{S}(4)$ is birational to $\mathcal{R}(4,7)$ via
the Cayley octad map $\gamma$.
The naive tropicalization of $\gamma$ 
has undesirable properties,
as the following example shows.

\begin{example}
  The naive tropicalization ${\rm trop}(\gamma)$ sends $q\in {\rm Dr}(4,7)$  to the vector $r$ given~by  
  $$ r_I \,\,= \,\,\begin{cases}
    q_I \quad\quad\quad\quad&\text{ if } 8\notin I, \\
    q_I - \sum_{i \in I} \mu_i &\text{ else,}
  \end{cases}$$ 
  where $\mu_1,\ldots,\mu_7$
  are tropicalizations of the Pl\"ucker binomials $x_1,\ldots,x_7$ in Theorem \ref{thm:gamma}.
However, this $r$ fails to lie in the Dressian ${\rm Dr}(4,8)$
for some choices of $q \in {\rm trop}({\rm Gr}(4,7))$.

A choice of $q$ where $r \not\in {\rm Dr}(4,8)$ is the tropical Pl\"ucker vector of the $4 \times 7$-matrix
$$ \begin{small}
    \begin{pmatrix}
      4 t   &  2 t^5&     -9 t^2 &      t^3     &  -5 t^4    & 6 t^4  &    -9 t \\
      9 t ^2&      8    &   3 t  &    4 t    & 3 t^3 &      -2 t^4 &       5  \\
      -5 t^4&     3 t^6&     -4 t^2&     3 t^4&      4 t^6&      8 t^4&       t^2   \\
      -3 t    &  -8 &    6 t^2&      -6 t^4  &   -t  &     -t^2  &    -9 t^3 
    \end{pmatrix}.
\end{small}
$$
Here $r = {\rm trop}(\gamma)(q)$ is not a valuated matroid
  because the `min achieved twice'-condition in the Pl\"ucker quadric
    $p_{1234}p_{1358} - p_{1235}p_{1348} + p_{1238}p_{1345}$  is violated.
Indeed, we have
  $r_{1234}=4$, $r_{1358}=18$, $r_{1235}=4 $, $r_{1348}=17$, $r_{1238}=15$ and $r_{1345}=5$, so that $\min\{22,21,20\}$ is not achieved twice.
  This is due to the fact that cancellations appear in the formula for computing $x_2$. Classically, we have $\val(x_2) = \text{val}(-p_{1234}p_{1256}p_{2357}p_{2467}+p_{1235}p_{1246}p_{2347}p_{2567})=21$. But tropically we obtain $\min \{ q_{1234}+q_{1256}+q_{2357}+q_{2467},q_{1235}+q_{1246}+q_{2347}+q_{2567}\} = \min\{20,20\}$.
\end{example}
  
More work is needed to correctly tropicalize the  birational isomorphism
$ \gamma: \mathcal{R}(4,7) \simeq \mathcal{S}(4)$. These
are very affine varieties of dimension $9$.
The aim would be to parametrize~the $9$-dimensional tropical variety
${\rm trop}(\mathcal{S}(4))$ whose points are the
tropical Cayley octads. Computing this space
is a challenge for the near future.
The following serves as point of departure.

  \begin{proposition}
  The tropical self-dual Grassmannian ${\rm trop}(\gamma({\rm Gr}(4,7))) = {\rm trop}({\rm SGr}(4,8))$
    is a balanced fan of pure dimension $16$ in $\RR^{70}/\RR {\bf 1}$ with $7$-dimensional
    lineality space. 
    This fan is contained in ${\rm trop}({\rm Gr}(4,8))^{\rm sd}$,
    which in turn satisfies the
     following strict inclusion:
    \begin{equation}
      \label{eq:threestrict}
      {\rm trop}({\rm Gr}(4,8))^{\rm sd} \, \subsetneq \,
      {\rm Dr}(4,8)^{\rm sd}.
\end{equation}      
    \end{proposition}
    
    \begin{proof}
  The Fundamental Theorem \cite[Theorem 3.2.3]{MS} yields the first statement.
We have
  \begin{equation} \label{eq:containmen}
  {\rm trop}({\rm SGr}(4,8)) \,\, \subseteq \, \,{\rm trop}({\rm Gr}(4,8))^{\rm sd} 
  \end{equation}
since the right hand side is the tropical prevariety of the
equations that cut out ${\rm SGr}(4,8)$. For the strictness of the inclusion (\ref{eq:threestrict}),
consider the  self-dual matroid {\tt 4.14.a} from Table \ref{tab:rank4}, which is not realizable over $\mathbb{C}$. Starting with this matroid, construct
    a point $q$ in  ${\rm Dr}(4,8)^{\rm sd}$ as follows: 
$q_I = 0$ if $I$ is a basis   of {\tt 4.14.a} and $q_I = 1$ otherwise. The matroid subdivision $\mathcal{M}(q)$ 
is self-dual. One of its cells is the matroid polytope of {\tt 4.14.a}.
This proves the claim as
$q \in \text{trop}({\rm Gr}(4,8))$ would imply that each matroid in $\mathcal{M}(q)$ is
  realizable over $\mathbb{C}$; see \cite{BS}.
 \end{proof}

\vspace{.2in}

\noindent {\bf Acknowledgments}.
Many colleagues helped us in this project.
We are very grateful~for interactions with
Daniele Agostini,
Tobias Boege, Edgar Costa, Marc H\"{a}rk\"{o}nen, Lukas K\"uhne, Mario Kummer, Marta Panizzut, Yue Ren, Angel David Rios Ortiz,  
and Benjamin Schr\"oter. We also thank the referees for their helpful suggestions which improved the paper.

\bigskip
 \bigskip

\noindent
\footnotesize
{\bf Authors' addresses:}

\smallskip

\noindent Alheydis Geiger,
MPI-MiS Leipzig
\hfill \url{alheydis.geiger@mis.mpg.de}

\noindent Sachi Hashimoto,
MPI-MiS Leipzig 
\hfill \url{sachi.hashimoto@mis.mpg.de}

\noindent Bernd Sturmfels,
MPI-MiS Leipzig  and UC Berkeley
\hfill \url{bernd@mis.mpg.de}

\noindent Raluca Vlad,
Brown University
\hfill \url{raluca_vlad@brown.edu}

\end{document}